\documentclass[12pt]{article}
\usepackage{epsfig, color}
\usepackage{amssymb,latexsym}
\usepackage{showlabels}
\usepackage{mathrsfs}
\usepackage{amsmath}

\newtheorem{obs}{Observation}[section]

\newtheorem{lem}[obs]{Lemma}
\newtheorem{theo}[obs]{Theorem}

\newtheorem{pro}[obs]{Proposition}

\newcounter{countcase}

\newcounter{countclaim}

\usepackage{bm}    %% For bold math symbols

\begin{document}
	
\baselineskip 0.7 cm

\title
{\bf\LARGE Spectral Sufficient Conditions for Graph Factors\footnote{This work is supported by the National Science Foundation of China (Nos.12261074 and 12461065).}\\[2mm]} 

\author{\small  Fengyun Ren,$^{a}$\ \ Shumin Zhang,$^{a,b,}$\footnote{Corresponding author.}\ \ Ke Wang, $^{a}$\\
                \small $^{a}$School of Mathematics and Statistics, \\
        \small Qinghai Normal University, Xining, 810001, China\\
		\small $^{b}$Academy of Plateau Science and Sustainability, People's  \\
		\small Government of Qinghai Province and Beijing Normal University\\[0.2cm]
		\small E-mails:  renfengyun1@aliyun.com,	zhangshumin@qhnu.edu.cn, wk950413@163.com.}
\date{}
\maketitle{}
\abstract{The $\{K_{1,1}, K_{1,2},C_m: m\geq3\}$-factor of a graph is a spanning subgraph whose each component
is an element of $\{K_{1,1}, K_{1,2},C_m: m\geq3\}$. In this paper, through the graph  spectral methods, we establish the lower bound of the signless Laplacian spectral radius and the upper bound of the distance spectral radius to determine whether a graph admits a $\{K_2\}$-factor.  We get a lower bound on the size (resp. the spectral radius) of
$G$ to guarantee that $G$ contains a $\{K_{1,1}, K_{1,2},C_m: m\geq3\}$-factor. Then we determine an upper
bound on the distance spectral radius of $G$ to ensure that $G$ has a $\{K_{1,1}, K_{1,2},C_m: m\geq3\}$-factor.
Furthermore, by constructing extremal graphs, we show that the above all bounds are best possible.\\[2mm]
{\bf Keywords:}  Signless Laplacian spectral radius; Distance spectral radius; A $\{K_{1,1}, K_{1,2},C_m: m\geq3\}$-factor; A $\{K_2\}$-factor; Size.\\[2mm]
{\bf (2020) Mathematics Subject Classification:}  05C42; 05C70.

\section{Introduction}

\begin{enumerate}
\item

 Many physical structures can be conveniently modeled by networks. Examples include a communication network
 with the nodes and links modeling cities and communication channels, respectively; and a railroad network with nodes
 and links representing railroad stations and railways between two stations, respectively. Factors and factorizations in
 networks are very useful in combinatorial design, network design, circuit layout and so on. In particular, a wide variety
 of systems can be described using complex networks. Such systems include: the cell, where we model the chemicals
 by nodes and their interactions by edges; the World Wide Web, which is a virtual network of Web pages connected
 by hyperlinks; and food chain webs, the networks by which human diseases spread, human collaboration networks
 etc \cite{N2003}. It is well known that a network can be represented by a graph. Vertices and edges of the graph correspond to
 nodes and links between the nodes, respectively. Henceforth we use the term “graph” instead of “network”.

 \quad\quad We study the factor problem in graphs, which can be considered as a relaxation of the well-known
 cardinality matching problem. The factor problem has wide-range applications in areas such as network
 design, scheduling and combinatorial polyhedra. For instance, in a communication network if we allow several large
 data packets to be sent to various destinations through several channels, the efficiency of the network will be improved
 if we allow the large data packets to be partitioned into small parcels. The feasible assignment of data packets can be
 seen as a factor flow problem and it becomes a fractional matching problem when the destinations and sources of
 a network are disjoint (i.e., the underlying graph is bipartite).

 \quad\quad  In this paper, we
only consider finite and undirected simple graphs. For graph theoretic notation and terminology not defined here, we refer to \cite{AK11,{B2014},{GR2001},YL09}. 
 A factor of a graph is a spanning subgraph. For integers a
 and b with $0 \leq a \leq b$, an $[a,b]$-factor is defined as a factor $F$ such that $a\leq \operatorname{deg}_F(x) \leq b$ for every vertex $x$, where
 $\operatorname{deg}_F(x)$ is the degree of $x$ in $F$. A $1$-factor is a $[1,1]$-factor. As usual, we denote by $C_\nu$, $K_\nu$, $K_{1,\nu-1}$ a cycle, a complete graph and
a star of order $\nu$, respectively.
 For positive integer $m\geq3$, A $\{K_{1,1}, K_{1,2}, C_3, C_4 \ldots C_m\}$-factor for which each component is (isomorphic to) one of $K_{1,1}$, $K_{1,2}$, $C_3$, $C_4$ $\ldots C_m$. A $\{K_{2}\}$-factor is a 1-factor.
 
 \quad\quad The existence of factors with given properties has received special attention; see, e.g. \cite{SS2021}. 
 In recent years, it is of great interest for researchers to find spectral sufficient conditions
such that a graph has a given factor; see the survey \cite{FLL2003}.

\quad\quad For two vertex disjoint graphs $G$ and $H$, and $G \cup H$ denotes the \emph{disjoint union} of $G$ and $H$, $G \vee H$
denotes the \emph{join} of $G$ and $H$, which is obtained from $G \cup H$ by adding all possible edges between any vertex
of $G$ and any vertex of $H$. For positive integer $k$ and a graph $G$, $kG$ denotes the graphs consisting of $k$ vertex
disjoint copies of $G$.

\quad\quad Given a graph, we denote by $\rho(G)$ the spectral radius of $G$, $\kappa(G)$ the signless Laplacian spectral radius
of $G$ and $\mu_1(G)$ the  distance spectral radius of $G$. 
Let graph $K(a; h_1\times q_1,h_2\times q_2, q_3,q_4,\ldots, q_h)=K_{a} \vee (h_1K_{q_1}\cup h_2K_{q_2}\cup K_{q_3}\cup K_{q_4}\ldots\cup K_{q_h})$. 
We are concerned about two types of factors with given properties.  

\quad\quad One is a $\{K_2\}$-factor, for which Amahashi \cite{A1985} gave a characterization. As far as we know, although this concept has been extended to various factors that contain a given edge, there is no spectral
sufficient condition for a graph has a $\{K_2\}$-factor.
Feng et al. \cite{FZLLLH2017}, Suil \cite{S2021} and Kimet al. \cite{KSS2022} gave spectral radius conditions that imply a graph on $\nu$ vertices has a matching of size (at least) $\frac{\nu-k}{2}$ for $0 \leq k \leq \nu$.
Instead of directly considering $\{K_2\}$-factor, we establish spectral radius
conditions that imply a graph on $\nu$ vertices has a matching of size (at least) $\frac{\nu-k}{2}$
for $0 \leq k \leq \nu$ containing any
given edge. We show the following results.

\quad\quad Our first main result gives a sufficient condition to ensure that a graph $G$
contains a $\{K_2\}$-factor according to the signless Laplacian spectral radius of $G$.
\begin{theo}\label{theorem3}
Let $G$ be a connected graph of order $\nu \geq max\{\delta(G) - 1 + (\delta(G))^2(\delta(G) + 1), \frac{19}{3}\delta(G)  + 3\}$, where $\nu$ is an even number and $\delta(G) \geq 19$.
If $\kappa(G)\geq \kappa(K(\delta(G); (\delta(G)+1)\times 1,\nu-2\delta(G)-1))$, then $G$ contains a $\{K_2\}$-factor, with equality if and only if  $G \cong K(\delta(G); (\delta(G)+1)\times 1,\nu-2\delta(G)-1)$.
\end{theo}

\quad\quad Our second main result gives a sufficient condition to ensure that a graph $G$
contains a $\{K_2\}$-factor according to the distance spectral radius of $G$.
\begin{theo}\label{theorem7}
Let $G$ be a connected graph of order $\nu \geq max\{\frac{2}{3}(\delta(G))^3, 12\delta(G)+5\}$, where $\nu$ is an even number and $\delta(G) \geq 12$. 
If $\mu_1(G)\leq \mu_1(K(\delta(G); ((\delta(G)+1)\times 1,\nu-2\delta(G)-1))$, then $G$ contains a $\{K_2\}$-factor, with equality if and only if $G \cong K(\delta(G); ((\delta(G)
+1)\times 1,\nu-2\delta(G)-1)$.
\end{theo}

\quad\quad The other is a $\{K_{1,1}, K_{1,2}, C_3, C_4 \ldots C_m\}$-factor, for which Akiyama J. and Era H. \cite{JH1980} gave a characterization. Also, we establish spectral conditions that imply a graph has a $\{K_{1,1}, K_{1,2}, C_3, C_4 \ldots C_m\}$-factor.

\quad\quad Our third main result gives a sufficient condition to ensure that a graph $G$
contains a $\{K_{1,1}, K_{1,2},C_m: m\geq3\}$-factor according to the size of $G$.
\begin{theo}\label{theorem11}
Let  $\nu\geq 4$ be a integer, and let $G$ be a connected graph with $\nu$ vertices.
If $|E(G)| \geq 11$ with $\nu = 7$, then $G$ contains a $\{K_{1,1}, K_{1,2},C_m: m\geq3\}$-factor, with equality if and only if $G \cong K(2; 5\times1)$.
If $|E(G)| \geq\binom{\nu-3}{2}+3$ with $\nu \neq 7$, then $G$ contains a
$\{K_{1,1}, K_{1,2},C_m: m\geq3\}$-factor, with equality if and only if $G \cong K(1 ; \nu-4, 3\times1)$.
\end{theo}
\quad\quad Our fourth main result gives a sufficient condition to ensure that a graph $G$
contains a $\{K_{1,1}, K_{1,2},C_m: m\geq3\}$-factor according to the spectral radius of $G$.
\begin{theo}\label{theorem12}
Let $\nu \geq 4$ be a integer, and let $G$ be a connected graph with $\nu$ vertices. Assume the largest root of
$x^3 - (\nu -5)x^2 -(\nu - 1)x+3\nu-15 = 0$ is $\beta(\nu)$.
If $\rho(G) \geq \frac{1+\sqrt41}{2}$ with $\nu = 7$, then $G$ contains a
$\{K_{1,1}, K_{1,2},C_m: m\geq3\}$-factor, with equality if and only if $G \cong K(2; 5\times1)$.
If $\rho(G) \geq \beta(\nu)$ with $\nu \neq 7$, then $G$ contains a
$\{K_{1,1}, K_{1,2},C_m: m\geq3\}$-factor, with equality if and only if $G \cong K(1 ; \nu-4, 3\times1)$.
\end{theo}
\quad\quad Our last main result gives a sufficient condition to ensure that a graph $G$
contains a $\{K_{1,1}, K_{1,2},C_m: m\geq3\}$-factor according to the distance spectral radius of $G$.
\begin{theo}\label{theorem13}
Let  $\nu \geq4$ be a integer, and let $G$ be a connected graph with $\nu$ vertices.
If $\mu_1(G) \leq \mu_1(K(2; 5\times1))$ with $\nu = 7$, then $G$ contains a
$\{K_{1,1}, K_{1,2},C_m: m\geq3\}$-factor, with equality if and only if $G \cong K(2; 5\times1)$.
If $\mu_1(G) \leq \mu_1(K(1 ; \nu-4, 3\times1))$ with $\nu \neq 7$, then $G$ contains a
$\{K_{1,1}, K_{1,2},C_m: m\geq3\}$-factor, with equality if and only if $G \cong K(1 ; \nu-4, 3\times1)$.
\end{theo}

\end{enumerate}
\section{Some preliminaries}

\begin{enumerate}
\item In this section, we present some necessary preliminary results, which will be used to prove our main results.

\quad Let $G$ be a graph with vertex set $V(G) = \{v_1, \ldots , v_\nu\}$ and edge set $E(G).$
For $v\in V(G)$, the neighborhood
$N_G(v)$ of $v$ is the set of vertices adjacent to $v$ in $G$, and the degree of $v$, denoted by $\operatorname{deg}_G(v)$, is the number
$|N_G(v)|$. For any $X \subseteq V(G)$, let $G[X]$ be the subgraph of $G$ induced by $X$, and write $G-X = G[V(G)\backslash X]$ if
$X\neq V(G)$. For convenience, we write $G-v$ for $G-\{v\}$. Let $E\subseteq E(G)$. Then, $G-E$ is the graph formed from $G$ by deleting the edges in $E$.
For convenience, we write $G-uv$ for $G-\{uv\}$.

 \quad 
The \emph{spectral radius} of a graph $G$ is the largest eigenvalue of the \emph{adjacency matrix} of $G$, which is defined as
the symmetric matrix $A(G)=(a_{ij})$, where $a_{ij}= 1$ if $v_i$ and $v_j$ are adjacent, and $a_{ij} = 0$ otherwise.
The \emph{signless Laplacian spectral radius} of a graph $G$ is the largest eigenvalue of its \emph{signless Laplacian matrix}
$K(G) = \operatorname{diag}(\deg_{G}(1), \deg_{G}(2), \ldots , \deg_{G}(\nu))+A(G)$, where $\operatorname{diag}(\deg_{G}(1), \deg_{G}(2), \ldots , \deg_{G}(\nu))$ is the degree diagonal matrix of $G$.

 \quad Let $G$ be a connected graph. The \emph{distance} between two vertices $v_i$ and $v_j$ in $G$, denoted
by $d_{ij}$, is the length of a shortest path from $v_i$ to $v_j$.
The \emph{distance matrix} $D(G)$ of a connected graph $G$ is defined
as $D(G) = (d_{ij})$. The distance matrix of $G$, denoted by $D(G)$, is
a $\nu \times \nu$ real symmetric matrix whose $(i, j)$-entry is $d_{ij}$. Then we can order the
eigenvalues of $D(G)$ as
 $$\mu_1(G) \geq \mu_2(G) \geq \cdots \geq \mu_\nu(G).$$ In view of the Perron-Frobenius theorem, $\mu_1(G)$ is always positive (unless $G$ is trivial)
and $\mu_1(G)\geq |\mu_j(G)|$ for $j = 2, 3, \cdots , \nu$, and we can call $\mu_1(G)$ the \emph{distance spectral
radius}.
\begin{lem}{\upshape (\cite{M1988})}\label{lem644}
Let $G$ be a connected graph with two nonadjacent vertices $u, v \in V(G)$. Then $\mu_1(G + uv) < \mu_1(G).$

\quad The following lemmas, following from the Perron-Frobenius theorem\cite{B2014, YYZL2017}, is well known.
\begin{lem}\label{lem4}
Let $G$ be a connected graph and let $H$ be a proper subgraph of $G$. Then $\kappa(H)<\kappa(G)$.
\end{lem}

\begin{lem}\label{lem444}
Let $G$ be a connected graph and let $H$ be a proper subgraph of $G$. Then $\rho(H)<\rho(G)$.
\end{lem}

\end{lem}

\quad\quad In mathematical literature, it is interesting to study graph factors.
For example,
Frobenius \cite{F1917} initiated the characterization of perfect matchings, who showed that an $\nu$-vertex bipartite graph contains a perfect matching if and only if the cardinality of each vertex cover is at least $\frac{\nu}{2}$.
The Hall Theorem provides a sufficient and necessary condition for a bipartite graph containing a perfect matching; see \cite{O2017}.
Tutte \cite{T1947} gave a necessary and sufficient condition for the existence of a $\{K_2\}$-factor in a graph, which is well known as Tutte's $1$-Factor Theorem. Bondy
\cite{B1976} gave a necessary and sufficient condition of a tree containing a $\{K_2\}$-factor.
In 2020, O \cite{S2021} showed that there is a close relationship between the spectral
radius and $\{K_2\}$-factor by Tutte's $1$-Factor Theorem. He established a sharp
upper bound on the number of edges (resp. spectral radius) of a graph without
a $\{K_2\}$-factor.
In 1953, Tutte \cite{T1953}  gave a necessary and sufficient condition for the existence of a $\{K_2, C_n\}$-factor in a graph. 
 Amahashi and Kano  \cite{AK82} gave a sufficient and necessary condition for the existence of a $\{K_{1,i} : 1\leq i\leq r\}$-factor with $r\geq2$.
 Egawa and Furuya \cite{YM2018}  gave a sufficient condition for the existence of a $\{P_2, C_{2i+1} : i \geq 2\}$-factor.
The distance eigenvalues of graphs have been studied for many years. For early work, see Graham
and Lov$\acute{a}$sz \cite{LL1978}. Ruzieh and Powers \cite{RP1990} and Stevanovi$\acute{c}$ and Ili$\acute{c}$ \cite{SI2010} showed that among n-vertex
connected graphs, the path $P_n$ is the unique graph with maximum distance spectral radius. For more advances along this line, we may
consult \cite{{YMA1997},{HJ1985},{S1974},{W1982}}.

\begin{lem}{\upshape (\cite{N2017})}\label{lem41} Let $G$ be a $\nu$-vertex connected graph and let $\mathbf{x} = (x_1, x_2, \ldots ,
x_\nu)^T$ be the Perron vector of $K(G)$ corresponding to $\kappa(G)$. If $v_i, v_j\in V(G)$
satisfy $N_G(v_i) \setminus \{v_j\} = N_G(v_j) \setminus \{v_i\}$, then $x_i = x_j$.
\end{lem}

\quad\quad Let $M$ be a real matrix of order $n$. And assume that $M$ can be described as
$$M=
\left(
  \begin{array}{ccc}
    M_{11} & \cdots & M_{1s} \\
    \vdots &  \ddots& \vdots\\
    M_{s1} & \cdots & M_{ss} \\
  \end{array}
\right)
$$
with respect to the partition $\pi : V = V_1\cup\cdots\cup V_s$, where $M_{ij}$ denotes the submatrix (block) of $M$ formed by rows in $V_i$
and columns in $V_j$. Let $m_{ij}$ denote the average row sum of $M_{ij}$. Then matrix $M_\pi = (m_{ij})$ is called the quotient matrix of $M$. If
for any $i$, $j$, $M_{ij}$ admits constant row sum, then $M_\pi$ is called the equitable matrix of $M$ and the partition is called equitable.
\begin{lem}{\upshape (\cite{{BH2012},LMWW2019})}\label{lem5}
 Let $M$ be a square matrix with an equitable partition $\pi$ and let $M_\pi$ be the corresponding quotient matrix.
Then every eigenvalue of $M_\pi$ is an eigenvalue of $M$. Furthermore, if $M$ is nonnegative and $M_\pi$ is irreducible, then the largest
eigenvalues of $M$ and $M_\pi$ are equal.
\end{lem}

\begin{lem}{\upshape (\cite{M1988})}\label{lem6}
Let $A$ be a real symmetric $n \times n$ matrix and let $B$ be a $m \times m$ principal submatrix of $A$ with $m < n$. Let
$\lambda_1(A) \geq \lambda_2(A) \geq \cdots \geq \lambda_n(A)$ be the eigenvalues of $A$, and $\lambda_1(B) \geq \lambda_2(B)\geq \cdots \geq\lambda_m(B)$ be the eigenvalues of $B$. Then for
$i = 1, \ldots , m,$
$$\lambda_i(A)\geq \lambda_i(B) \geq \lambda_{n-m+i}(A).$$
\end{lem}

\quad\quad The following lemma can be easily derived by the Rayleigh quotient \cite{HJ1985}.
\begin{lem}\label{lem9}
Let $G$ be a connected graph with order $\nu$. Then
$$\mu_1(G) = \max_{\mathbf{x}\neq\mathbf{0}}\frac{\mathbf{x}^T D(G)\mathbf{x}}{\mathbf{x}^T\mathbf{x}}
\geq \frac{\mathbf{1}^TD(G)\mathbf{1}}{\mathbf{1}^T \mathbf{1}}
= \frac{2W(G)}{\nu}
,$$
where $\mathbf{1} = (1, 1,\ldots , 1)^T$ and $W(G) =\sum\limits_{i<j}d_{ij}$.
\end{lem} 

\quad\quad Inspired by  A. Amahashi \cite{A1985}, we establish two spectral sufficient conditions for a graph to be a $\{K_2\}$-factor, respectively.
Inspired by Akiyama J. \cite{JH1980}, we get three spectral sufficient conditions
for a graph to be a $\{K_{1,1}, K_{1,2},C_m: m\geq3\}$-factor, respectively.

\quad\quad The rest of this paper is arranged as follows: In Section 3
we present a $\{K_2\}$-factor. In Section 4
we present a $\{K_{1,1}, K_{1,2},C_m: m\geq3\}$-factor.
In Section 5 we present some extremal graphs.
\end{enumerate}

\section{A $\{K_2\}$-factor}

\begin{enumerate}
\item {\bf Signless Laplacian spectral radius}

	 We give the proof of Theorem \ref{theorem3}, which gives a sufficient
condition via the signless Laplacian spectral radius of a connected graph to ensure that the graph contains a
$\{K_2\}$-factor.
In 1985, Amahashi \cite{A1985} gave a sufficient and necessary condition for the existence of a $\{K_2\}$-factor.
\begin{lem}{\upshape (\cite{A1985})}\label{lem3.1}
 Let $G$ be a graph. Then $G$  contains a
$\{K_2\}$-factor if and only if for every set
$X \subseteq V (G)$, $\operatorname{o}(G -X) \leq |X|$, where $\operatorname{o}(G -X)$ is the number of odd components(components with odd order) in a graph $G -X$.
\end{lem}

\begin{lem}\label{lem3}
Let $G$ be a connected graph. If $G$
does not have a $\{K_2\}$-factor, then there exists a non-empty subset
$X \subseteq V (G)$ such that $\operatorname{o}(G-X) > |X|$. Furthermore, if $|V (G)|$ is even,
then
$$
\begin{aligned}
\operatorname{o}(G- X)&\equiv |X|&&(mod&2)
\end{aligned}
$$
and $$\operatorname{o}(G- X) \geq |X| + 2.$$
\end{lem}
\begin{proof}
By Lemma \ref{lem3.1}, it is easy to see that there exists a subset $X \subseteq V (G)$
satisfying $\operatorname{o}(G - X) > |X|$. Assume $|V (G)|$ is even. Then $\operatorname{o}(G- X)$ is
odd (resp. even) if and only if $|X|$ is odd (resp. even), which is equivalent to
say that $|X|$ is odd (resp. even). Hence, $\operatorname{o}(G-X) \geq|X| + 2$. This completes the proof of Lemma \ref{lem3}.
\end{proof}

\quad\quad Now we give a proof of Theorem \ref{theorem3}.

\begin{proof}
 We prove it by contradiction. $G$ has no $\{K_2\}$-factor. By Lemma \ref{lem3},  there
exists some nonempty subset  $X\subseteq V(G)$ satisfying $$\operatorname{o}(G-X)\geq |X| + 2.$$ 
 Then $G$ is a spanning subgraph of $$H_1 = K(|X|; \nu_1,\nu_2, \cdots, \nu_{|X|+2})$$ for some positive odd integers
$\nu_1\geq\nu_2 \geq \cdots \geq \nu_{|X|+2}$ with $\sum\limits_{i=1}^{|X|+2}\nu_i = \nu-|X|$. That is to say, $\kappa(G) \leq \kappa(H_1)$, where the equality holds if and only if $G\cong H_1$. 

\quad\quad In order to characterize the structure of $G$, we need the following fact.

{\bf Fact 1.}  Let $\nu = \sum\limits_{i=1}^{s}\nu_i+|X|$. If $\nu_1 < \nu-|X|-r(s-1)$ and $\nu_1 \geq \nu_2 \geq \cdots \nu_s \geq r \geq 1$, then 
$$\kappa(K(|X|, \nu-|X|-r(s-1), (s-1)\times r))>\kappa(K(|X|; \nu_1, \nu_2, \cdots , \nu_s )).$$

\begin{proof}
Let $G =K(|X|; \nu_1, \cdots , \nu_s )$ and $$H_2 = K(|X|; \nu-|X|-r(s-1), (s-1)\times r).$$ Consider the
partition: $V(G) = V(K_{|X|}) \cup V(K_{\nu_1}) \cup V(K_{\nu_2}) \cup \cdots \cup V(K_{\nu_s} )$, where 
$V(K_{|X|}) = \{u_1, u_2,\ldots , u_{|X|}\}$, and 
$V(K_{\nu_j}) = \{w_{j_1}, w_{j_2}, \ldots , w_{j_{\nu_i}}\}$
for $1\leq j \leq s.$

Let 
$$
\begin{aligned}
E_1=& \{w_{j_l}w_{j_q}|2 \leq j \leq s,1 \leq l \leq \nu_j-r,  \nu_j-r + 1  \leq q \leq \nu_j \},
\end{aligned}
$$
$$E_2=\{w_{1_l}w_{j_i}| 1 \leq l \leq \nu_1, 2 \leq j \leq s, 1\leq i \leq \nu_j-r\}$$ and
 $$E_3 = \{w_{j_i}w_{l_q}|2 \leq j \leq s-1, j+1 \leq l \leq s; 1\leq i \leq \nu_j-r, 1 \leq q \leq \nu_l-r\}.$$ 
Obviously, $G + E_1-E_2-E_3\cong H_2$.

\quad\quad Suppose $\mathbf{z} = (z_1, \ldots , z_\nu)^T$ is 
the Perron vector of $K (G)$, and let $z_j$ denote the entry of $\mathbf{z}$ corresponding to the vertex $v_j \in V(G)$. 
 By Lemma \ref{lem41},
one has $z_a = z_b$ for all $v_a, v_b$ in $V (K_{|X|})$ (resp. $V(K_{\nu_c})$ for $1\leq c\leq s)$.
For convenience, 
let $z_p = x_1$ for all $v_p \in V(K_{|X|})$ and $z_p = y_j$ for all $v_p \in V(K_{\nu_j})$ where $1 \leq j \leq s$.
Together with  $K (G)\mathbf{x} = \kappa(G)\mathbf{x},$ we get
$$\kappa(G)y_1 = |X| x_1 + (2\nu_1 +|X|-2)y_1,$$
$$\kappa(G)y_i = |X| x_1 + ( 2\nu_i+|X|-2)y_i, (i = 2, \ldots , s).$$
Thus, $$[\kappa(G)-( 2\nu_1+|X| -2)]y_1 = [\kappa(G)-(2\nu_i+|X|-2)]y_i,$$ where $2 \leq i \leq s.$ Note that $K_{|X|+\nu_1}$ and $K_{|X|+\nu_2}$ are two proper
subgraphs of $G$. Combining with Lemma \ref{lem4}, we obtain $\kappa(G) > \kappa(K_{|X|+\nu_1}) = 2(\nu_1+|X|-1)$ and $\kappa(G) > \kappa(K_{|X|+\nu_i}) = 2(\nu_i+|X| - 1).$ By $\nu_1 \geq \nu_2 \geq \cdots\geq \nu_s$, we have $y_1 \geq y_i$ for $2 \leq i \leq s$. By the Rayleigh quotient, we obtain
$$
\begin{aligned}
\kappa(H_2) -\kappa(G)\geq& \mathbf{z}^T( K (H_2)- K (G))\mathbf{z}\\
=& 2(\sum_{i=2}^{s}(\nu_1y_1- ry_i)(\nu_i-r)y_i +\sum_{j=2}^{s-1}\sum_{i=j+1}^{s}(\nu_j- r)(\nu_i- r)y_i y_j)\\
&+\sum_{i=2}^{s}((\nu_1- 2r)y_i^2+\nu_1y_1^2) +\sum_{j=2}^{s-1}\sum_{i=j+1}^{s}(\nu_j- r)(\nu_i- r)( y_i^2+ y_j^2)\\
>& 0.
\end{aligned}
$$
 This completes the proof of Fact 1.
\end{proof}

We denote minimum degree $\delta(G) = \operatorname{min}\{\operatorname{deg}_G(v) : v \in V (G)\}$.
We proceed by the following three possible cases.

{\bf Case 1.} $|X| < \delta(G).$ 

\quad\quad In this case, let $$H_3 = K(|X|; \nu-(|X|+1)(\delta(G)-|X|+1)-|X|, (|X|+1)\times (\delta(G)-|X|+1)).$$ Recall that $G$ is a spanning subgraph
of $H_1 = K(|X|; \nu_1, \nu_2, \cdots, \nu_{|X|+2})$, where $\nu_1 \geq \nu_2 \geq \cdots \geq \nu_{|X|+2}$ and $|X|+\sum\limits_{i=1}^{|X|+2}\nu_i = \nu$. 
Note that $\delta(G)\leq\delta(H_1)$, one obtains that $\nu_{|X|+2} \geq \delta -|X|+ 1$. 
Together with Fact 1, we obtain $q(H_3)\geq q(H_1)$, with equality if and only
if $(\nu_1, \nu_2, \ldots , \nu_{|X|+2}) = (\nu -(|X|+1)(\delta(G)-|X| + 1)- |X|, \delta(G)-|X| + 1, \ldots , \delta(G)-|X| + 1)$.

\quad\quad In what follows, we consider $|X| \geq 1$. 
 Assume that $\kappa(H_3) \geq 2[\nu - (|X|+1)(\delta(G) -|X|+ 1)]$. Suppose $\mathbf{x} = (x_1, \ldots , x_\nu)^T$ is
the Perron vector of $K (G_3)$, and let $x_j$ denote the entry of $\mathbf{x}$ corresponding to the vertex $v_j \in V(G_3)$. 
By Lemma \ref{lem41},
one has $x_a = x_b$ for all $v_a, v_b$ in $V (K_{|X|})$ (resp. $V((|X|+1)K_{\delta(G)-|X|+1})$ and $V(K_{\nu-(|X|+1)(\delta(G)-|X|+1)-|X|}))$.
For convenience, let $x_p = f$ for all $v_p \in V(K_{|X|})$, $x_p = h$ for all $v_p \in V((|X|+1)K_{\delta(G)-|X|+1})$, and $x_p = m$ for all $v_p \in V(K_{\nu-(|X|+1)(\delta(G)-|X|+1)-|X|})$.
 Together with  $K (H_3)\mathbf{x} = \kappa(H_3)\mathbf{x},$ we have
$$
\begin{aligned}
\kappa(H_3)f=& (\nu + |X| - 2)f + (|X|+1)(\delta(G) - |X|+ 1)h \\
&+ [\nu- |X| - (|X|+1)(\delta(G) - |X|+ 1)]m,\\
\kappa(H_3)h=& |X|f + (2\delta(G)- |X|)h,\\
\kappa(H_3)m=& |X|f + [2\nu - 2(|X|+1)(\delta(G) - |X| + 1)- |X|- 2]m.
\end{aligned}
$$
It is straightforward to check that
$$h =\frac{|X|f}{\kappa(H_3) - (2\delta(G) - |X|)}$$ and
$$m =\frac{|X|f}{\kappa(H_3) - [2\nu - 2(|X|+1)(\delta(G) - |X|+ 1)- |X|- 2]}.$$
Recall that $\nu \geq (\delta(G))^2(\delta(G) + 1)+\delta(G) -1$. Hence, $$\kappa(H_3) \geq 2[\nu- (|X|+1)(\delta(G)-|X| + 1)] > 2(\delta(G) + 1).$$ So we have
$$
\begin{aligned}
\kappa(H_3) + 2=& \nu + |X| + \frac{|X|(|X|+1)(\delta(G) - |X|+ 1)}{\kappa(H_3)- 2\delta(G) + |X|}\\
&+\frac{|X|[\nu - |X|- (|X|+1)(\delta(G)- |X|+ 1)]}{\kappa(H_3) - [2\nu- 2(|X|+1)(\delta(G) - |X|+ 1) - |X|- 2]}\\
<& \nu + |X| + \frac{|X|(|X|+1)(\delta(G) - |X|+ 1)}{|X| + 2}\\
&+ \frac{|X|(\nu - |X| - (|X|+1)(\delta(G) - |X|+ 1))}{|X| + 2}\\
=& \nu + |X| + \frac{|X|(\nu - |X|)}{|X| + 2}\\
=& 2\nu - 2(|X|+1)(\delta(G) - |X|+ 1) \\
&- \frac{2\nu- 2|X| - 2(|X|+1)(|X| + 2)(\delta(G) - |X|+ 1)}{|X| + 2}\\
\leq& 2\nu - 2(|X|+1)(\delta(G) - |X|+ 1)\\
\leq& \kappa(G_3),
\end{aligned}
$$a contradiction.
Thus, 
$$
\begin{aligned}
\kappa(H_3)& < 2[\nu - (|X|+1)(\delta(G) - |X|+ 1)]\\
&= 2[\nu -\delta(G) -2 -(|X|(\delta(G) - |X|) - 1)]\\
&\leq 2(\nu - \delta(G)  - 2).
\end{aligned}
$$
Let $$H_4= K(\delta(G); \nu-2\delta(G)-1 , (\delta(G)+1)\times 1).$$
 Note that $K_{\nu-\delta(G)-1}$ is a proper subgraphs of $H_4$.
According to Lemma \ref{lem4}, we have $\kappa(H_4) > \kappa(K_{\nu-\delta(G)-1}) = 2(\nu - \delta(G) - 2)$. So we obtain $\kappa(H_3) < \kappa(H_4)$. Combining with  $\kappa(G) \leq \kappa(H_1) \leq \kappa(H_3)$, we get $\kappa(G) < \kappa(H_4)$, a contradiction.

{\bf Case 2.} $|X| = \delta(G)$. 

\quad\quad In this case, according to Fact 1 we obtain $\kappa(H_1) \leq \kappa(K(\delta(G); \nu-2\delta(G)-1, (\delta(G)+1)\times 1))$, with equality
if and only if $H_1 \cong K(\delta(G); \nu-2\delta(G)-1, (\delta(G)+1)\times 1)$. Together with  $\kappa(G) \leq \kappa(H_1)$, we get $\kappa(G) \leq
\kappa(K(\delta(G); \nu-2\delta(G)-1, (\delta(G)+1)\times 1))$, with equality if and only if $G \cong K(\delta(G); \nu-2\delta(G)-1, (\delta(G)+1)\times 1)$. Note that
$K(\delta(G); \nu-2\delta(G)-1, (\delta(G)+1)\times 1)$ has no a $\{K_2\}$-factor. Hence, we can obtain a contradiction.

{\bf Case 3.} $|X|\geq \delta(G) + 1$. 

\quad\quad In this case, let $$H_5 = K(|X|; \nu-2|X|-1 , (|X|+1) \times 1).$$ According to Fact 1, we get $\kappa(H_1)\leq \kappa(H_5)$,
with equality if and only if $(\nu_1, \nu_2,\ldots , \nu_{|X|+2}) = (\nu-2|X|-1, 1, \ldots , 1)$. 

\quad\quad Consider the partition $\pi_1$: $V(H_5) = V(K_{|X|}) \cup V(K_{\nu-2|X|-1}) \cup V((|X|+1)K_1)$. One has the quotient matrix of $K(H_5)$ corresponding to
the partition $\pi_1$ as

$$M_{1}=
\left(
  \begin{array}{ccc}
    \nu + |X|-2 & \nu-2|X|-1 & |X|+1 \\
    |X| & 2\nu-3|X|-4 & 0 \\
    |X| & 0 & |X| \\
  \end{array}
\right).
$$
Let $M$ be a square matrix. Then its \emph{characteristic polynomial} is denoted by $\Phi_M(x) = det(x I-M)$, where $I$ is the identity
matrix, whose order is the same as that of $M$.
Then the characteristic polynomial of $M_1$ equals
$$
\begin{aligned}
\Phi_{M_1}(x)=& x^3 + (|X|-3\nu +6)x^2 - [4|X|^2 - (\nu -4)|X|  - 2(\nu^2 - 4\nu+ 4]x\\
&-2|X|^3+ 2(2\nu  -5)|X|^2 - 2(\nu^2- 5\nu +6)|X|.
 \end{aligned}
$$
Since the partition $\pi_1$: $V(H_5) =V(K_{|X|}) \cup V(K_{\nu-2|X|-1}) \cup V((|X|+1)K_1)$ is equitable, by Lemma \ref{lem5} the largest root, say $\kappa(H_5)$, of $\Phi_{M_1}(x) = 0$.

\quad\quad Recall that $H_4= K(\delta(G); \nu-2\delta(G)-1 , (\delta(G)+1)\times 1)$. Consider the partition $\pi_2$: $V(H_4) = V(K_{\delta(G)}) \cup V(K_{\nu-2\delta(G)-1}) \cup V((\delta(G)+1) K_1)$. Then
the quotient matrix of $K (H_4)$ corresponding to the partition $\pi_2$ is given as
$$M_2=
\left(
  \begin{array}{ccc}
    \nu + \delta(G)-2 & \nu-2\delta(G)-1 & \delta(G)+1 \\
    \delta(G) & 2\nu-3\delta(G)-4 & 0 \\
    \delta(G) & 0 & \delta(G) \\
  \end{array}
\right).
$$
Then the characteristic polynomial of $M_2$ equals
$$
\begin{aligned}
\Phi_{M_2}(x)=& x^3 - (3\nu - \delta(G) - 6)x^2 + [2\nu^2 + (2\nu + \delta(G)- 8)\nu - 4(\delta^2  + \delta- 2)]x\\
&- 2[(\delta(G) \nu^2- ( 2(\delta(G))^2 + 5\delta(G))\nu   + (\delta(G))^3+5(\delta(G))^2+ 6\delta(G)].\\
\end{aligned}
$$
Since the partition $\pi_2$: $V(H_4) = V(K_{\delta(G)}) \cup V(K_{\nu-2\delta(G)-1}) \cup V((\delta(G)+1) K_1)$ is equitable, by Lemma \ref{lem5}, the largest
root, say $\kappa(H_4)$, of $\Phi_{M_2}(x)= 0$.

\quad\quad According to Lemma \ref{lem4}, we have $\kappa(H_4) > \kappa(K_{\nu-\delta(G)-1}) = 2(\nu- \delta(G) -2)$. 
In what follows, we are to show $\kappa(H_5) < \kappa(H_4)$. So it suffices to prove
$\Phi_{M_1}(x)-\Phi_{M_2}(x) > 0$ for $x \geq 2(\nu - \delta(G)-2)$. By a simple calculation, one has
$$\Phi_{M_1}(x)-\Phi_{M_2}(x) = (|X|- \delta(G))g_1(x).$$Let 
$$
\begin{aligned}
g_1(x) =& x^2+(\nu-4\delta(G)-4|X|-4)x-2|X|^2\\
&+2(2\nu-\delta(G)-5)|X|-2[\nu^2-(2\delta(G)+5)\nu+(\delta(G))^2+5\delta(G)+6]
\end{aligned}
$$
be a real function in $x$ for $x \in [2(\nu - \delta(G)-2), +\infty)$. It is routine to check that the derivative function of $g_1(x$) is
$$g_1'(x) = 2x+\nu- 4\delta(G)-4|X| -4.$$
Thus,
the unique solution of $g_1'(x) = 0$ is $\frac{4\delta(G)+4|X|+4-\nu}{2}$.  Bearing in mind that $\nu \geq \frac{19}{3} \delta(G)+3$ and $\nu \geq 2(|X| + 1)$. It is straightforward
to check that
$$
\begin{aligned}
2(\nu-\delta(G) - 2) -\frac{4\delta(G)+4|X|+4-\nu }{2}&= -2|X| + \frac{5\nu}{2} -4\delta(G) - 6\\
&\geq\frac{3\nu}{2} -4\delta(G) - 4\\
&\geq \frac{11\delta(G)+ 1 }{2}\\
&> 0.
\end{aligned}
$$
Hence, $g_1(x)$ is a increasing  function in the interval $[2(\nu -\delta(G) -2), +\infty)$. So we have
$$
\begin{aligned}
g_1(x)\geq& g_1(2(\nu -\delta(G)  -2))\\
=& -2|X|^2 -2(2\nu  - 3\delta(G) - 3)|X| + 4\nu^2 -2(7\delta(G) +9)\nu  \\
&+ 10[(\delta(G))^2 + 3\delta(G)+2] .
\end{aligned}
$$
Recall that $\nu \geq 2(|X| + 1).$ We obtain $|X| \leq \frac{\nu-2}{2}.$
Let $$g_2(x) = -2x^2 -2(2\nu  - 3\delta(G) - 3)x + 4\nu^2 -2(7\delta(G) +9)\nu + 10[(\delta(G))^2 + 3\delta(G)+2] $$ be a real function in $x$ for 
$x \in [\delta(G) + 1, \frac{\nu-2}{2}]$. In fact, the derivative function of $g_2(x)$ is
$$g_2'(x) = -4x -2(2\nu  - 3\delta(G) - 3).$$
Thus, $\frac{-2\nu+3\delta(G)+3}{2}$
is the unique solution of $g_{2}'(x) = 0$. Together with  $\nu \geq \frac{19}{3}\delta(G)+3$ and $\delta(G) \geq 19$, we get
$$
\begin{aligned}
\frac{-2\nu + 3\delta(G) + 3}{2}- (\delta(G) + 1) &= \frac{-2\nu+1 + \delta(G) }{2}\\
&\leq -\frac{35\delta(G)}{6} -\frac{5}{2}\\
&< 0.
\end{aligned}
$$
Hence, $g_2(x)$ is a decreasing function in the interval $[\delta(G) + 1, \frac{\nu-2}{2}]$. Therefore, we obtain
$$
\begin{aligned}
g_2(|X|) &\geq g_2(\frac{\nu-2}{2})\\
&= \frac{3\nu^2}{2}- (11\delta(G)+9) \nu  + 2[5(\delta(G))^2 + 12\delta(G) +6].
\end{aligned}
$$
Let $$g_{3}(x)=\frac{3x^2}{2}- (11\delta(G)+9) x  + 2[5(\delta(G))^2 + 12\delta(G) +6]$$ be a real function in $x$ for $x \in [\frac{19\delta(G)}{3}+3, +\infty)$.
It is routine to check that the derivative function of  $g_3(x)$ is
$$g_3'(x) = 3x  -11\delta(G)-9.$$

Hence the unique solution of $g_3'(x) = 0$ is $\frac{11\delta(G)+9}{3}$. As
$$\frac{19\delta(G)+9}{3}> \frac{11\delta(G)+9 }{3},$$
one obtains that $g_3(x)$ is increasing for $x \in[\frac{19\delta(G)+9}{3}, +\infty)$. 
Hence, we get
$$g_3(\nu)\geq g_3(\frac{19\delta(G)+9}{3}) = \frac{1}{2}[(\delta(G))^2 - 18\delta(G)- 3].$$
Let $$g_4(x) = \frac{1}{2}(x^2 - 18x- 3)$$ be a real function in $x$ for $x \in [19, +\infty)$. In fact, the derivative function of $g_4(x)$ is
$$g_4'(x) = x -9.$$
Observe that  $g_4'(x) \geq g_4'(19) = 10 > 0$ in the interval $[19, +\infty)$. Clearly, $g_4(x)$ is increasing for $x \in [19, +\infty)$. Therefore, 
$g_4(\delta(G))\geq g_4(19) = 8 > 0$. 
Hence, when $x\geq 2(\nu -\delta(G) -2)$, we have $\Phi_{M_1}(x)-\Phi_{M_2}(x) >0$. Thus, we get
$\kappa(H_5) < \kappa(H_4)$.  Together with $\kappa(G) \leq \kappa(H_1) \leq \kappa(H_5)$, we obtain $\kappa(G) < \kappa(H_4)$, a contradiction.

\quad\quad This completes the proof of Theorem \ref{theorem3}.

\end{proof}

	\item {\bf Distance spectral radius}

We give the proof of Theorem \ref{theorem7}, which gives a sufficient
condition via the distance spectral radius of a connected graph to ensure that the graph contains a
$\{K_2\}$-factor.

\quad\quad Now we give a proof of Theorem \ref{theorem7}.

\begin{proof}
  We prove it by contradiction. $G$ has no a $\{K_2\}$-factor. By Lemma \ref{lem3},  there
exists some nonempty subset  $X\subseteq V(G)$ satisfying $$\operatorname{o}(G-X)\geq |X| + 2.$$

 Then $G$ is a spanning subgraph of $$H_1 = K(|X|; \nu_1,\nu_2, \cdots, \nu_{|X|+2})$$ for some positive odd integers
$\nu_1\geq\nu_2 \geq \cdots \geq \nu_{|X|+2}$ with $|X|+\sum\limits_{i=1}^{|X|+2}\nu_i = \nu$. That is to say,  $\mu_1(G) \geq \mu_1(H_1)$, where the equality holds if and only if $G\cong H_1$. 

\quad\quad In order to characterize the structure of $G$, we need the following fact.

{\bf Fact 2.}  Let $\nu = |X| + \sum\limits_{i=1}^{s}\nu_i$. If $\nu_1 < \nu-|X|-r(s-1)$ and $\nu_1 \geq \nu_2 \geq \cdots \nu_s \geq r \geq 1$, then 
$$\mu_1(K(|X|; \nu_1, \nu_2, \cdots, \nu_s )) >\mu_1(K(|X|; \nu-|X|-r(s-1), (s-1)\times r)).$$

\begin{proof}
Let $G = K(|X|; \nu_1 \nu_2, \cdots , \nu_s )$ and $$H_2 =K(|X|; \nu-|X|-r(s-1), (s-1)\times r).$$ Consider the partition: 
$V(G) = V(K_{|X|}) \cup V(K_{\nu_1}) \cup V(K_{\nu_2}) \cup \cdots \cup V(K_{\nu_s} )$, where $V(K_{|X|}) = \{u_1, u_2,\ldots , u_{|X|}\}$, and 
$V(K_{\nu_j}) = \{w_{j_1}, w_{j_2}, \ldots , w_{j_{\nu_i}}\}$
for $1\leq j \leq s.$

Let 
$$
\begin{aligned}E_1=& \{w_{j_l}w_{j_q}|2 \leq j \leq s,1 \leq l \leq \nu_j-r,  \nu_j-r + 1  \leq q \leq \nu_j \},
\end{aligned}
$$
$$E_2=\{w_{1_l}w_{j_i}| 1 \leq l \leq \nu_1, 2 \leq j \leq s, 1\leq i \leq \nu_j-r\}$$ and
 $$E_3 = \{w_{j_i}w_{l_q}|2 \leq j \leq s-1, j+1 \leq l \leq s; 1\leq i \leq \nu_j-r, 1 \leq q \leq \nu_l-r\}.$$ 
Obviously,  $G + E_1- E_2- E_3  \cong H_2$.

\quad\quad Suppose $\mathbf{x} = (x_1, \ldots , x_\nu)^T$ is
the Perron vector of $D (H_2)$, and let $x_i$ denote the entry of $\mathbf{x}$ corresponding to the vertex $v_i \in V(H_2)$.
By Lemma \ref{lem41}, one has $x_a = x_b$ for all $v_a, v_b$ in $V (K_{|X|})$ (resp. $V((s -1)K_r)$ and $V(K_{\nu-|X|-r(s-1)})$). 
 For convenience, let $x_i = z_1$ for all $v_i \in V(K_{|X|})$, $x_i =z_2 $ for all $v_i \in V((s -1)K_r)$, and $x_i = z_3$ for all $v_i \in V(K_{\nu-|X|-r(s-1)})$.
Together with $D (H_2)\mathbf{x} = \mu_1(H_2)\mathbf{x},$ we obtain
$$\mu_1(H_2)z_2 = |X|z_1 + (2r s- 3r -1)z_2 + 2[\nu-|X| -r(s- 1)]z_3,$$
$$\mu_1(H_2)z_3 = |X|z_1 + 2r(s - 1)z_2 + [\nu - |X| -1 - r(s- 1)]z_3.$$
Thus, $$(\mu_1(H_2) + r + 1)z_2 = (\mu_1(H_2) + \nu- |X| -rs + r + 1)z_3.$$ So we get $$z_2 = (1 +\frac{\nu-|X|-rs}{\mu_1(H_2)+r+1})z_3.$$ By the
Rayleigh quotient, we obtain
$$
\begin{aligned}
\mu_1(G) -\mu_1(H_2)\geq& \mathbf{x}^T( D (G)- D (H_2))\mathbf{x}\\
=&  \nu_1\sum_{i=2}^{s}(\nu_i-r)z_3^2 + (\nu_2 - r)[(\nu- |X|- \nu_2 - (s - 2)r)z_3^2- 2rz_2z_3] \\
&+ \cdots
+ (\nu_s- r)[(\nu- |X| - \nu_s -(s- 2)r)z_3^2- 2rz_2z_3]\\
=& (\nu_2 -r)[(\nu + \nu_1- |X| -\nu_2- (s- 2)r)z_3^2- 2rz_2z_3] + \cdots\\
&+ (\nu_s - r)[(\nu + \nu_1 - |X|- \nu_s -(s- 2)r)z_3^2-2rz_2z_3].\\
\end{aligned}
$$
In what follows, we are to show $\mu_1(G) > \mu_1(H_2)$. Then it suffices to prove $$ (\nu + \nu_1-|X|- \nu_2 - (s -2)r)z_3^2-2rz_2z_3 > 0.$$ Observe that
$K_{\nu-r(s-1)}$ is a subgraph of $H_2$. According to Lemma \ref{lem6}, we have
$$\mu_1(H_2) > \mu_1(K_{\nu-r(s-1)}) = \nu - 1 - r(s - 1).$$
Hence,
$$
\begin{aligned}
&(\nu + \nu_1-|X|- \nu_2 - (s -2)r)z_3^2-2rz_2z_3\\
&= [\nu_1- \nu_2 + (\nu-|X| -rs)-\frac{2r(\nu-|X| - rs)}{\mu_1(H_2) + r + 1}]z_3^2\\
&=\frac{z_3^2}{\mu_1(H_2) + r + 1}
[(\nu_1-\nu_2)(\mu_1(H_2) + r + 1) + (\nu -|X|
-rs)(\mu_1(H_2)- r + 1)]\\
&>\frac{z_3^2}{\mu_1(H_2) + r + 1}[(\nu_1-\nu_2)(\mu_1(H_2) + r + 1) + (\nu-|X| - rs)(\nu- rs)]\\
&> 0.
\end{aligned}
$$
Therefore, we get $\mu_1(G) > \mu_1(H_2)$. 
 This completes the proof of Fact 2.

\end{proof}

We proceed by the following three possible cases.

{\bf Case 1.}  $|X| < \delta(G).$ 

In this case, let $$H_3 = K(|X|; (\nu-(\delta(G)-|X|+1)(|X|+1)-|X|, (|X|+1)\times (\delta(G)-|X|+1)).$$ Recall that $G$ is a spanning subgraph
of $H_1 = K(|X|; \nu_1, \nu_2, \cdots, \nu_{|X|+2})$, where $\nu_1 \geq \nu_2 \geq \cdots \geq \nu_{|X|+2}$ and $\sum\limits_{i=1}^{|X|+2}\nu_i = \nu-|X|$. Clearly, 
Note that $\delta(H_1) \geq \delta(G)$, one obtains that
$\nu_{|X|+2} \geq \delta(G) + 1-|X|$. Together with Fact 2, we get $\mu_1(H_1)\geq \mu_1(H_3)$, with equality if and only
if $(\nu_1, \nu_2, \ldots , \nu_{|X|+2}) = (\nu -(|X|+1)(\delta(G)-|X| + 1)-|X|,\delta(G) -|X|+ 1, \ldots , \delta(G) -|X|+ 1)$.

\quad\quad In what follows, we consider $|X| \geq 1$.
Consider the partition $\pi_3$: $V(H_3) = V(K_{|X|}) \cup V(K_{\nu-(|X|+1)(\delta(G)-|X|+1))-|X|} \cup V((|X| + 1)K_{\delta(G)-|X|+1})$. 
Then the quotient matrix of
$D(H_3)$ corresponding to the partition $\pi_1$ equals
 $$M_3=
\left(
  \begin{smallmatrix}{}
   |X|-1 &&&&&&&& \nu- (|X| + 1)(\delta(G)-|X|+ 1)-|X| &&&&&&&& (|X| + 1)(\delta(G) -|X+ 1|) \\
   |X| &&&&&&&& \nu- (|X| + 1)(\delta(G)-|X|+ 1)-|X|-1 &&&&&&&& 2(|X|+ 1)(\delta(G) -|X|+ 1) \\
   |X| &&&&&&&& 2(\nu- (|X| + 1)(\delta(G)-|X|+ 1)-|X|) &&&&&&&& \delta(G) +2|X|(\delta(G) -|X|+ 1)-|X| \\
  \end{smallmatrix}
\right).
$$
Its characteristic polynomial is 
$$
\begin{aligned}
\Phi_{M_3}(x)=& x^3 + [|X|^2 - ( \delta(G) +1)|X| -\nu  + 3]x^2
+ [2|X|^4 - (4\delta(G) + 2)|X|^3\\
&+(2\nu 
 + 2(\delta(G))^2-3\delta(G) -3)|X|^2 - [( 2\delta(G) - 1)\nu- 5(\delta(G))^2 -5\delta(G)]|X|\\
 &  -(3\delta(G)+5)\nu  + 3(\delta(G))^2
+ 6\delta(G) + 6]x- |X|^5 + ( 2\delta(G) + 3)|X|^4 \\
&-(\nu +(\delta(G))^2+ 3\delta(G))|X|^3
+ [(  2 + \delta(G))\nu - 6\delta(G) - 5]|X|^2\\ 
&- [(\delta(G) - 2)\nu
- 4(\delta(G))^2 - 4\delta(G)]|X|  \\
&-( 3\delta(G)+4)\nu
 + 3(\delta(G))^2 + 6\delta(G) + 4.
 \end{aligned}
$$
Since the partition $\pi_3$: $ V(H_3) = V(K_{|X|})\cup V(K_{\nu-(|X|+1)(\delta(G)-|X|+1)-|X|}) \cup V((|X|  + 1)K_{\delta(G)-|X|+1})$ is equitable, by Lemma \ref{lem5}, the largest root, say  $\mu_1(H_3)$,  of $\Phi_{M_3}(x)= 0$.

Let $$H_4 = K(\delta(G); \nu-2\delta(G)-1, (\delta(G)+1)\times 1).$$ Consider the partition $\pi_2$:  $V(H_4) = V(K_{\delta(G)})\cup V(K_{\nu-2\delta(G)-1})\cup V((\delta(G)+1)K_1)$.
Then the quotient matrix of $D(H_4)$
corresponding to the partition $\pi_2$ equals
$$M_4=
\left(
  \begin{array}{ccc}
    \delta(G)-1 & \nu-2\delta(G)- 1 & \delta(G)  + 1 \\
   \delta(G) & \nu- 2\delta(G)- 2 &  2(\delta(G)  + 1)\\
   \delta(G)  & 2(\nu-2\delta(G) -1) &  2\delta(G) \\
  \end{array}
\right).
$$
Thus the characteristic polynomial of $M_4$ is
$$
\begin{aligned}
\Phi_{M_4}(x)=& x^3 + (3 -\delta(G) -\nu)x^2 + [-( 2\delta(G)+5)\nu   + 5(\delta(G))^2 + 6\delta(G) + 6]x\\
&+ [  (\delta(G))^2- \delta(G) -4]\nu -2(\delta(G))^3 + 2(\delta(G))^2 + 6\delta(G) + 4.
 \end{aligned}
$$
Since the partition $\pi_2$:  $V(H_4) = V(K_{\delta(G)})\cup V(K_{\nu-2\delta(G)-1}) \cup V((\delta(G)+ 1)K_1)$ is equitable, by Lemma \ref{lem5}, the largest root, say $\mu_1(H_4)$, of $\Phi_{M_4}(x)= 0$. Observe that $\mu_1(H_4)> \mu_1(K_{\nu-\delta(G)-1}) = \nu-\delta(G)- 2.$

\quad\quad In what follows, we are to show  $\mu_1(H_3) > \mu_1(H_4)$. Then it suffices to prove $\Phi_{M_3}(\mu_1(H_4))< 0$. Together with $\Phi_{M_4}(\mu_1(H_4)) = 0$, we
obtain
$$
\begin{aligned}
\Phi_{M_3}(\mu_1(H_4))=& \Phi_{M_3}(\mu_1(H_4))-\Phi_{M_4}(\mu_1(H_4))\\
=& -(\delta(G)-|X|)[(|X|-1)(\mu_1(H_4))^2 + (2|X|^3  -( 2\delta(G) + 2)|X|^2 \\
&+ (2\nu 
 - 5\delta(G)- 3)|X| + \nu   + 2\delta(G))\mu_1(H_4)\\
 &- |X|^4+ ( \delta(G)+ 3)|X|^3-\nu  |X|^2 + ( 2\nu   
-6\delta(G)- 5)|X| \\
&+ ( \delta(G) + 2)\nu - 2(\delta(G))^2 - \delta(G)].
 \end{aligned}
$$
Let 
$$
\begin{aligned}g_5(x) =& (|X|-1)x^2+(2|X|^3-(2\delta(G)+2)|X|^2+(2\nu-5\delta(G)-3)|X|+\nu+2\delta(G))x\\
&-|X|^4+(\delta(G)+3)|X|^3-\nu|X|^2+(2\nu-6\delta(G)-5)|X|+(\delta(G)+2)\nu\\
&-2(\delta(G))^2-\delta(G)
 \end{aligned}
$$be a real function in $x$ for $x \in(\nu - \delta(G)-2, +\infty).$  It is routine to check that the derivative function of $g_5(x)$ is
$$g_5'(x) = 2(|X|-1)x + 2|X|^3 - (2\delta(G)+2)|X|^2 + (2\nu-5\delta(G)-3)|X| + \nu  + 2\delta(G).$$
Bearing in mind that  $\nu\geq \frac{2}{3}(\delta(G))^3$ and $1 \leq |X| \leq\delta(G)-1$. It is straightforward to check that
$$
\begin{aligned}
g_5'(x)&> g_5'(\nu-\delta(G) - 2)\\
&= 2|X|^3 -( 2\delta(G) + 2)|X|^2 + (4\nu -7\delta(G) -7)|X| -\nu
  + 4\delta(G) + 4\\
&\geq2 - ( 2\delta(G) + 2)(\delta(G) -1)^2 + (4\nu  -7\delta(G) - 7)
-\nu   + 4\delta(G) + 4\\
&= 3\nu -\delta(G)   + 2(\delta(G))^2 -2(\delta(G))^3- 3\\
&\geq 2(\delta(G))^2 -\delta(G) - 3\\
&> 0.
\end{aligned}
$$
Hence, $g_5(x)$ is a increasing function in the interval $(\nu-\delta(G)-2, +\infty)$. So we have
$$
\begin{aligned}
g_5(\mu_1(H_4))>& g_5(\nu- \delta(G)-2)\\
=&-|X|^4 + (2\nu-\delta(G) -1)|X|^3 + [-(2\delta(G) +3)\nu  + 2(\delta(G))^2 \\
&+ 6\delta(G) + 4]|X|^2+ [3\nu^2 - ( 9\delta(G)+9)\nu + 6(\delta(G))^2 + 11\delta(G) \\
&+ 5]|X|
+ (  4\delta(G) + 4)\nu- 5(\delta(G))^2 - 9\delta(G)- 4\\
\geq& -(\delta(G) -1)^4 + (2\nu -\delta(G) - 1) + [-(2\delta(G)+ 3)\nu + 2(\delta(G))^2 \\
&+ 6\delta(G)
+ 4](\delta(G)-1)^2 + [3\nu^2 - (9\delta(G) + 9)\nu  + 6(\delta(G))^2 \\
&+ 11\delta(G) + 5]
+ (   4\delta(G)+ 4)\nu- 5(\delta(G))^2-9\delta(G) -4\\
=& 3\nu^2  -[2(\delta(G))^3 - (\delta(G))^2 + \delta(G)+6]\nu \\
& + (\delta(G))^4 + 6(\delta(G))^3 - 11(\delta(G))^2+ 3\delta(G) + 3.
\end{aligned}
$$
Let $$g_6(x) = 3x^2 -[2(\delta(G))^3 -(\delta(G))^2 +\delta(G) + 6]x  + (\delta(G))^4 + 6(\delta(G))^3-11(\delta(G))^2 + 3\delta(G) + 3$$ be a real function in $x$ for $x \in [\frac{2}{3}
(\delta(G))^3, +\infty)$. In fact, the derivative function of $g_6(x)$
is $$g_6'(x) = 6x -2(\delta(G))^3 + (\delta(G))^2 -\delta(G)-6.$$
Thus,  $\frac{1}{6}[2(\delta(G))^3-(\delta(G))^2 + \delta(G) + 6]$ is the unique solution of 
$g_6'(x) = 0$. Obviously,
$$\frac{2}{3}\delta^3(G) >\frac{1}{6}(  2(\delta(G))^3-(\delta(G))^2 + \delta(G) + 6).$$
Hence, $g_6(x)$ is a increasing function in the interval
 $ [\frac{2}{3}(\delta(G))^3 , +\infty)$. Therefore, we
obtain
$$
\begin{aligned}
g_6(\nu)&\geq g_6(\frac{2}{3}(\delta(G))^3)\\
&= \frac{2}{3}(\delta(G))^5 + \frac{1}{3}
(\delta(G))^4 +    2(\delta(G))^3   -11(\delta(G))^2
+3\delta(G)   + 3\\
&\geq \frac{2}{3}
(\delta(G))^5+  \frac{1}{3}(\delta(G))^4 +2(\delta(G))^3 - 11(\delta(G))^2-5\delta(G)   + 3\\
&> 0.
\end{aligned}
$$
Thus,  $g_5(\mu_1(H_4)) > 0$ and $\Phi_{M_3}(\mu_1(H_4)) < 0$. Therefore, $\mu_1(H_3) > \mu_1(H_4)$. Together with $\mu_1(G) \geq \mu_1(H_1) \geq \mu_1(H_3)$, we
get $\mu_1(G) > \mu_1(H_4)$, a contradiction to the condition.

{\bf Case 2.}  $|X| = \delta(G)$. 

\quad\quad In this case, according to  Fact 2 we get $\mu_1(H_1) \geq \mu_1(K(\delta(G); \nu-2\delta(G)-1, (\delta(G)+1)\times 1))$, with equality
if and only if $H_1 \cong K(\delta(G); \nu-2\delta(G)-1, (\delta(G)+1)\times 1)$. Together with  $\mu_1(G) \geq \mu_1(H_1)$, we obtain $\mu_1(G) \geq
\mu_1(K(\delta(G); \nu-2\delta(G)-1, (\delta(G)+1)\times 1))$, with equality if and only if $G \cong K(\delta(G); \nu-2\delta(G)-1, (\delta(G)+1)\times 1)$. Note that
$K(\delta(G); \nu-2\delta(G)-1, (\delta(G)+1)\times 1)$ has no a $\{K_2\}$-factor, a
contradiction.

{\bf Case 3.}  $|X|\geq \delta(G) + 1$. 

\quad\quad In this case, let $$H_5 = K(|X|; \nu-2|X|-1, (|X|+1)\times 1).$$ According to  Fact 2, we obtain $\mu_1(H_1)\geq\mu_2(H_5)$,
with equality if and only if $(\nu_1, \nu_2, \ldots ,\\ \nu_{|X|+2}) = (\nu-2|X|-1, 1, \ldots , 1)$. 

Consider the partition $\pi_4$: $V(H_5) = V(K_{\nu-2|X|-1}) \cup V(K_{|X|})\cup V((|X|+1)K_1)$. 
One has the quotient matrix of $D(H_5)$ corresponding to the partition $\pi_4$ as
$$M_5=
\left(
  \begin{array}{ccc}
    \nu-2|X| -2 & |X| & 2(|X|+1) \\
    \nu-2|X|-1 & |X|-1 & |X|+1 \\
    2(\nu-2|X|-1) & |X| & 2|X| \\
  \end{array}
\right).
$$
Then the characteristic polynomial of $M_5$ equals
$$
\begin{aligned}
\Phi_{M_5}(x)=& x^3 - (\nu+|X|-3)x^2 + [5|X|^2 -( 2\nu-6)|X| -5\nu + 6]x\\
&-2|X|^3 + (\nu + 2)|X|^2 - (\nu  - 6)|X| -4\nu + 4.
 \end{aligned}
$$
Since the partition $\pi_4$: $V(H_5) = V(K_{\nu-2|X|-1})\cup V(K_{|X|}) \cup V((|X|+1)K_1)$ is equitable, by Lemma \ref{lem5}, the largest root, say $\mu_1(H_5)$), of $\Phi_{M_5}(x) = 0$.

\quad\quad Recall that $H_4= K(\delta(G); \nu-2\delta(G)-1, (\delta(G)+1)\times 1)$. Consider the partition $\pi_5$: $V(H_4) = V(K_{\nu-2\delta(G)-1}) \cup V(K_{\delta(G)}) \cup V((\delta(G)+1) K_1)$.
Then the quotient matrix of $D(H_4)$ corresponding to the partition $\pi_5$ is given as
$$M_6=
\left(
  \begin{array}{ccc}
    \nu - 2\delta(G)-2 & \delta(G) & 2(\delta(G)+1) \\
    \nu-2\delta(G)-1 & \delta(G)-1 & \delta(G)+1 \\
    2(\nu-2\delta(G)-1) & \delta(G) & 2\delta(G) \\
  \end{array}
\right).
$$
Then the characteristic polynomial of $M_6$ equals
$$
\begin{aligned}
\Phi_{M_6}(x)= & x^3 - (\nu+\delta(G)-3)x^2 - [(2\delta(G)+5)\nu  - 5(\delta(G))^2-6\delta(G) - 6]x\\
&+ ( (\delta(G))^2 - \delta(G)-4)\nu - 2(\delta(G))^3 + 2(\delta(G))^2 + 6\delta(G) + 4.
\end{aligned}
$$
Since the partition $\pi_5$: $V(H_4) = V(K_{\nu-2\delta(G)-1}) \cup V(K_{\delta(G)}) \cup V((\delta(G)+1) K_1)$ is equitable. By Lemma \ref{lem5}, the largest
root, say $\mu_1(H_4)$,  of $\Phi_{M_6}(x)= 0$.

It is obvious that $\mu_1(H_4)> \mu_1(K_{\nu-\delta(G)-1}) = \nu- \delta(G)-2$. 
In what follows, we are to show $\mu_1(H_5) > \mu_1(H_4)$. So it suffices to prove 
$\Phi_{M_5}(\mu_1(H_4)) < 0$. 
Observe that $\Phi_{M_6}(\mu_1(H_4))= 0$. Thus,
$$
\begin{aligned}
\Phi_{M_5}(\mu_1(H_4))=&\Phi_{M_5}(\mu_1(H_4))-\Phi_{M_6}(\mu_1(H_4))\\
=& (\delta(G)-|X|)[(\mu_1(H_4))^2 + (2\nu-5\delta(G)-5|X|-6)\mu_1(H_4) + 2|X|^2 \\
&- (\nu- 2\delta(G)+ 2)|X|- (\delta(G) - 1)\nu  + 2(\delta(G))^2- 2\delta(G)- 6].
\end{aligned}
$$
Let $g_7(x) = x^2 + (2\nu- 5\delta(G)- 5|X|-6)x + 2|X|^2 - (\nu - 2\delta(G)+2)|X| - ( \delta(G) - 1)\nu  + 2(\delta(G))^2 - 2\delta(G)-6$ be a real function in $x$ for $x\in(\nu-\delta(G)  -2, +\infty)$. It is routine to check that the
derivative function of $g_7(x)$ is
$$g_7'(x) = 2x + 2\nu  -5\delta(G)-5|X|-6.$$
Therefore, $\frac{1}{2}(5\delta(G) +5|X| + 6- 2\nu)$ is the unique solution of $g_7'(x) = 0$. Combining with $\nu \geq 12\delta(G)+ 5$ and $\nu \geq 2|X| + 2$,
we obtain
$$
\begin{aligned}
(\nu -\delta(G) - 2)-\frac{1}{2}(5\delta(G) +5|X| + 6-2\nu) &= \frac{1}{2}(4\nu -7\delta(G)-5|X|-10)\\
&\geq \frac{1}{4}(3\nu  - 14\delta(G) - 10)\\
&\geq\frac{1}{4}(22\delta(G) + 5)\\
&> 0.
\end{aligned}
$$
Thus, $g_7(x)$  is a increasing function in the interval $(\nu-\delta(G)  -2, +\infty)$. Hence, we have $g_7(\mu_1(H_4)) > g_7(\nu- \delta(G)  -2) =
2|X|^2 + (8 + 7\delta(G)-6\nu)|X| + 3\nu^2  -(10\delta(G)+13)\nu  + 8(\delta(G))^2 + 18\delta(G) + 10.$ 
Let $$g_8(x) =
2x^2 + (8 + 7\delta(G) - 6\nu)x + 3\nu^2  - ( 10\delta(G) +13)\nu + 8(\delta(G))^2 + 18\delta(G) + 10$$ be a real function in
$x$ for $x \in [\delta(G) + 1,\frac{\nu-2}{2}]$. We may obtain the derivative function of $g_8(x)$ as
$$g_8'(x) = 4x -6\nu + 7\delta(G) + 8.$$
Therefore, $\frac{1}{4}(6\nu - 7\delta(G) -8)$ is the unique solution of $g_8'(x) = 0.$ Obviously,
$$\frac{6\nu -7\delta(G) -8}{4}>\frac{ \nu - 2}{2}.$$
Then $g_8(x)$ is decreasing in the interval $ [\delta(G) + 1, \frac{\nu-2}{2}]$. Thus,
$g_8(|X|) \geq g_8( \frac{\nu-2}{2}) =
\frac{1}{2}\nu^2 - (\frac{13\delta(G)}{2}+5)\nu + 8(\delta(G)) + 11\delta(G) + 4$. Let $$g_9(x) = \frac{1}{2}
x^2 - ( \frac{13\delta(G)}{2}+5)x   + 8(\delta(G))^2 + 11\delta(G) + 4$$ be a real function
in $x$ for $x \in [12\delta(G)  + 5, +\infty)$. In fact, the derivative function of $g_9(x)$
is
$$g_9'(x) = x -\frac{13\delta(G)}{2}- 5.$$
Thus, $\frac{13\delta(G)}{2} + 5$ is the unique solution of $g_9'(x) = 0$. Obviously,
$$\frac{13\delta(G)}{2} + 5 < 12\delta(G) + 5.$$
Therefore, $g_9(x)$ is a increasing function in the interval $[12\delta(G)  + 5, +\infty)$. Hence, we
get
$$g_9(\nu) \geq g_9(12\delta(G)  + 5) = 2(\delta(G))^2 -\frac{43}{2}\delta(G)  - \frac{17}{2}.$$

Let $$g_{10}(x) = 2x^2  - \frac{43}{2}x- \frac{17}{2}$$
be a real function in $x$ for $x \in [12, +\infty)$. It is routine to check that the derivative
function of $g_{10}(x)$ is
$$g_{10}'(x) = 4x  - \frac{43}{2}.$$
Thus, $\frac{ 43}{8}$ is the unique solution of $g_{10}'(x) = 0$. As
$$\frac{ 43}{8} < 12.$$
one obtains that $g_{10}(x)$ is increasing for $x \in[12, +\infty)$. Therefore, we have
$$g_{10}(x) \geq g_{10}(12) = \frac{43}{2}> 0.$$
Hence, $g_7(\mu_1(H_4)) > 0$ and $\Phi_{M_5}(\mu_1(H_4)) < 0$. So we get $\mu_1(H_5) > \mu_1(H_4)$. Together with $\mu_1(G) \geq \mu_1(H_1) \geq \mu_1(H_5)$, we
obtain $\mu_1(G) > \mu_1(H_4)$, a
contradiction.

\quad\quad This completes the proof of Theorem\ref{theorem7}.

\end{proof}

\end{enumerate}

\section{A $\{K_{1,1}, K_{1,2}, C_{m}: m\geq3\}$-factor}
\begin{enumerate}
	\item {\bf Size}

We give the proof of Theorem \ref{theorem11}, which gives a sufficient condition via the size of a connected graph $G$ to ensure that $G$ is a $\{K_{1,1}, K_{1,2}, C_{m}: m\geq3\}$-factor.
 We denote by $I(G)$ the set of isolated vertices of $G$, and let $\operatorname{iso}(G) = |I(G)|$.  $G$ is called \emph{trivial} if $|V(G)|=1$. 
In 1980, Akiyama J. and Era H. \cite{JH1980} gave a sufficient and necessary condition for the existence of a $\{K_{1,1}, K_{1,2},C_m: m\geq3\}$-factor.
\begin{lem}{\upshape (\cite{JH1980})}\label{lem333}
A graph $G$ has a $\{K_{1,1}, K_{1,2},C_m: m\geq3\}$-factor if and only if
$$\operatorname{iso}(G-X)\leq2|X|$$ for all  $X\subseteq V(G)$.
\end{lem}

\quad\quad Now we give a proof of Theorem \ref{theorem11}.

\begin{proof}
 We prove it by contradiction. $G$ is not a $\{K_{1,1}, K_{1,2}, C_{m}: m\geq3\}$-factor. 
Then by Lemma \ref{lem333}, there
exists some nonempty subset $X\subseteq V(G)$ such that $\operatorname{iso}(G-X)\geq 2|X| + 1$. 

Suppose that $X=\emptyset$, we have $\operatorname{iso}(G)\geq 1$, which contradicts the condition that $G$ is a connected graph of $\nu \geq4$.

Choose a connected graph $G$ of order $\nu$ such that its size is as large as possible. Based on the choice of $G$, the induced subgraph $G[X]$ and each connected component of
$G-X$ are complete graphs, respectively. Moreover, $G =G[X]\vee(G-X)$.

{\bf Fact 3.}  {$G-X$ has at most  one non-trivial connected component. }

\begin{proof}
Otherwise, adding edges among  all non-trivial connected component results in a larger order non-trivial connected component, contrary to the choice of $G$. This completes the proof of Fact 3.
\end{proof}
   
We distinguish the following two cases to prove.

{\bf Case 1.}  $G_1$ is only one non-trivial connected component of $G-X$.

\quad\quad In this case, let $|V(G_1)| =\nu_1\geq2$. In what follows, we will prove that $\operatorname{iso}(G-X) =2|X| + 1$. Suppose that $\operatorname{iso}(G-X)\geq 2|X| + 2$, then we create a new graph $H_6$ obtained from $G$ by connecting every vertex  of $G_1$ to one vertex of $I(G-X)$ with an edge. Hence, we obtain $|E(G)|< |E(G)| + \nu_1=|E(H_6)|$. Together with $\operatorname{iso}(H_6- X) \geq 2|X| + 1$, we have a contradiction to the choice of $G$. Therefore, we get $\operatorname{iso}(G-X) \leq 2|X| + 1$. Combining with
$\operatorname{iso}(G-X) \geq 2|X| + 1$, we have $\operatorname{iso}(G-X) = 2|X| + 1$ and $G =K(|X|; \nu_1, (2|X| + 1)\times 1)$.
Observe that $\nu = 3|X| + \nu_1+ 1 \geq 3|X| + 3 \geq 6$ and $|E(G)| =\binom{\nu-2|X|-1}{2}+(2|X| + 1)|X|=: h (\nu, |X|)$.

{\bf Fact 4.} Let $\nu \geq 3|X| + 3$, $|X|$ be  two positive integers, and let $h(\nu, |X|) =\binom{\nu-2|X|-1}{2}+(2|X| + 1)|X|$. Then
$ h(\nu, 1)\geq h(\nu, |X|)$.

\begin{proof}
If $|X| = 1$, then $h (\nu, |X|)=h(\nu, 1)$. We aim to demonstrate that  the result holds for $|X| \geq2.$  Through a simple calculation, we derive
$$
\begin{aligned}h (\nu, |X|)-h(\nu, 1) &=\binom{\nu-2|X| - 1}{2} +(2|X| + 1)|X| -\binom{\nu - 3}{2}-3\\
&=2(|X|- 1)(2|X|  - \nu  + 4). 
\end{aligned}
$$Next we prove $2|X|  - \nu  + 4  < 0$. Together with  $|X| \geq 2$ and $\nu \geq 3(|X| + 1)$, we have
$$2|X|  - \nu  + 4  \leq 2|X|  - 3(|X| + 1)  + 4
=1-|X|
\leq - 1
<0.$$
This completes the proof of Fact 4.
\end{proof}

According to Fact 4, it follows that $|E(G)| \leq h (\nu, 1)$ for $\nu \geq 6$.
Next, we consider $\nu = 7$.
By a simple calculation, we get $h (7, 1) = 9 < 11$. Hence, we obtain a contradiction to the
condition.

{\bf Case 2.} $G - X$ is no non-trivial connected component.

\quad\quad In what follows, we will show that $\operatorname{iso}(G-X) \leq 2|X| + 2$. Suppose that $\operatorname{iso}(G-X) \geq 2|X| + 3$, then  we get a new graph $H_{7}$ obtained from $G$ by adding an edge
of $I(G-X)$. Hence, we obtain $H_{7}-X$ has only one non-trivial connected component and $\operatorname{iso}(H_{7}- X)\geq 2|X| + 1$. Thus, we have
$ |E(H_{7})|>|E(G)|$, contrary to the choice of $G$.

Observe that  $\operatorname{iso}(G-X) \geq 2|X|+1$. Next, consider that $\operatorname{iso}(G-X) = 2|X|+1$ and $\operatorname{iso}(G-X) = 2|X|+2$.
We proceed by considering the following two subcases.

{\bf Subcase 2.1.} $\operatorname{iso}(G-X) = 2|X| + 1$. 

\quad\quad In this subcase $G = K(|X|; (2|X| + 1)\times 1)$, $\nu = 3|X| + 1 \geq4$ and $|E(G)| =\binom{|X|}{2}+(2|X| + 1)|X|$.

If $|X| = 2$, then $\nu = 7$ and $|E(G)| = 11$. 
If $|X| \neq 2 (\nu = 4$ or $\nu = 3|X| + 1 \geq 10)$, we obtain
$$h (\nu, 1) -|E(G)| = \binom{3|X|-2}{2} +3-\binom{|X|}{2}-(2|X| + 1)|X|= 2(|X| -1)(|X|- 3) \geq 0.$$

Thus, we get $|E(G)| \leq h (\nu, 1) =\binom{\nu-3}{2}+3$, a contradiction.

{\bf Subcase 2.2.} $\operatorname{iso}(G-X) = 2|X| + 2$.

\quad\quad In this subcase $G = K(|X|; (2|X| + 2)\times 1)$ , $\nu = 3|X| + 2$ and $|E(G)| = \binom{|X|}{2}+(2|X| + 2)|X|.$  Hence, we get
$$h (\nu, 1) - |E(G)| =\binom{3|X|  - 1}{2} +3 -\binom{|X|}{2}-(2|X| + 2)|X|= 2(|X|-1)(|X|-2).$$
If $|X| = 1$, then $|E(G)| = h (\nu, 1) =\binom{\nu-3}{2}+3$. 
If $|X| = 2$, then $|E(G)| =h (\nu, 1) =\binom{\nu-3}{2}+3$.

Next, we will show that $h (\nu, 1)-|E(G)| >0$ for $|X| \geq 3$. Since $|X|\geq 3$, we obtain
$2(|X|-1)(|X|-2)>0.$
Thus, we get $|E(G)| < h(\nu, 1) =\binom{\nu-3}{2}+3$, a contradiction.

\quad\quad This finishes the proof of Theorem \ref{theorem11}.

\end{proof}

	\item {\bf Spectral radius}

We give the proof of Theorem \ref{theorem12}, which gives a sufficient condition via the spectral radius of a connected graph $G$ to ensure that $G$ is a $\{K_{1,1}, K_{1,2}, C_{m}: m\geq3\}$-factor.

\quad\quad Now we give a proof of Theorem \ref{theorem12}.

\begin{proof}
 We prove it by contradiction. $G$ is not a $\{K_{1,1}, K_{1,2}, C_{m}: m\geq3\}$-factor. 
Then by Lemma \ref{lem333}, there
exists some nonempty subset $X\subseteq V(G)$ such that $\operatorname{iso}(G-X)\geq 2|X| + 1$. 

Suppose that $X=\emptyset$, we have $\operatorname{iso}(G)\geq 1$, which contradicts the condition that $G$ is a connected graph of $\nu \geq4$.

Choose a connected graph $G$ of order $\nu$ such that its adjacency spectral radius  is as large as possible. Based on the choice of $G$, the induced subgraph $G[X]$ and each connected component of
$G-X$ are complete graphs, respectively. Moreover, $G =G[X]\vee(G-X)$.

{\bf Fact 5.}  {$G-X$ has at most  one non-trivial connected component. }

\begin{proof}
Otherwise, adding edges among  all non-trivial connected component results in a larger order non-trivial connected component, contrary to the choice of $G$. This completes the proof of Fact 5.
\end{proof}

   Let $\alpha(x) =x^3 - (\nu -5)x^2 -(\nu - 1)x+3\nu-15$ be a real function in $x$ and let $\beta(\nu)$ be the largest root of $\alpha(x) = 0.$
We distinguish the following two cases to prove.

{\bf Case 1.}   $G_1$ is only one non-trivial connected component of $G-X$.

\quad\quad In this case, let $|V(G_1)| =\nu_1\geq2$.  In what follows, we will prove that  $\operatorname{iso}(G-X) =2|X| + 1$.  Suppose that $\operatorname{iso}(G-X)\geq 2|X| + 2$, then we  create  a new graph $H_6$ obtained from $G$ by connecting every vertex of $G_1$ with one vertex of $I(G-X)$ by an edge. 

Hence, we obtain that a proper
spanning subgraph of $H_6$ is $G$ and $\operatorname{iso}(H_6-X) \geq 2|X|+1$. Together with  Lemma \ref{lem444}, we have $\rho(G) < \rho(H_6)$, contrary to the choice of $G$. Thus, we get $\operatorname{iso}(G-X) \leq 2|X| + 1$.

Observe that $\operatorname{iso}(G-X)\geq 2|X| + 1$. Then  $\operatorname{iso}(G-X) = 2|X| + 1$, $\nu = 3|X| + \nu_1+ 1 \geq 3(|X| + 1) \geq 6$ and $G = K(|X|; \nu_1, (2|X| + 1)\times1)$.
According to the partition
 $V(G) = V(K_{|X|}) \cup V((2|X| + 1)K_1) \cup V(K_{\nu_1})$, the quotient matrix of $A(G)$
is
$$M_7 =\left(
       \begin{array}{ccc}
          |X| -1 & 2|X| + 1 & \nu - 3|X|- 1 \\
         |X| & 0 & 0\\
         |X| & 0 & \nu- 3|X| -2 \\
       \end{array}
     \right).
      $$
Then, the characteristic polynomial of $M_7$ equals
$$\Phi_{M_7}(x) = x^3- (\nu-2|X| -3)x^2-(2|X|^2+ \nu - |X| -2)x -6|X|^3 + (2\nu- 7)|X|^2 + (\nu  - 2)|X|.$$
Since $V(G) = V(K_{|X|})\cup V((2|X|+ 1)K_1)\cup V(K_{\nu_1})$ is an equitable partition, together
with Lemma \ref{lem5}  we get that the largest root of $\Phi_{M_7}(x) = 0$, denoted as $\beta_1$, equals $\rho(G)$.

Observe that $K(|X|; (2|X|+1)\times 1)$ is a proper subgraph of $G$. Consider the partition  $V(K(|X|; (2|X|+1)\times 1)) = V(K_{|X|})\cup V((2|X|+1)K_1)$.  The corresponding quotient matrix of $A(K(|X|; (2|X|+1)\times 1))$ is denoted
as
$$M_8=\left(
       \begin{array}{cc}
          |X| -1 & 2|X| + 1  \\
         |X| &  0\\
       \end{array}
     \right).
      $$
Then the characteristic polynomial of $M_8$ is given by
$$\Phi_{M_8}(x) = x^2 -(|X|- 1)x-(2|X| + 1)|X|.$$
Since $V(K(|X|; (2|X|+1)\times 1)) = V(K_{|X|}) \cup V((2|X| + 1)K_1)$ is an equitable partition, according to
 Lemma \ref{lem5}, we get that the largest root of $\Phi_{M_8}(x) = 0$, say
$\beta_2$, equals $\rho(K(|X|; (2|X|+1)\times 1))$. Hence,
we obtain
$$\rho(K(|X|; (2|X|+1)\times 1))=\beta_2=\frac{\sqrt{9|X|^2+2|X|+1}+|X|-1}{2}. $$

According to Lemma \ref{lem444}, $ \beta_2<\beta_1$. To show $\beta(\nu)\geq\rho(G)$, it suffices to prove $\alpha(\beta_1) \leq 0$.
Together with $\Phi_{M_7}(\beta_1) = 0$, we obtain
$$\alpha(\beta_1) = \alpha(\beta_1)-\Phi_{M_7}(\beta_1)  = (|X|-1)f_1,$$
where $f_1 = -2\beta_1^2 + (2|X| + 1)\beta_1 + 6|X|^2 - ( 2\nu-13)|X|   - 3\nu     + 15$. Clearly, $\alpha(\beta_1) = 0$ for
$|X| = 1$. So, a root of $\alpha(x)$ is $\beta_1$ and  $\rho(G) = \beta_1 \leq \beta(\nu)$. Next, we  will prove $f_1<0$ for $|X| \geq 2$.

Together with $\nu \geq 3(|X| + 1),$ we have
$$
\begin{aligned}
f_1& =-(2|X| + 3)\nu - 2\beta_1^2 + (2|X| + 1)\beta_1 + 6|X|^2 + 13|X| + 15\\
&\leq -3(2|X| + 3)(|X| + 1)-2\beta_1^2 + (2|X| + 1)\beta_1+ 6|X|^2 + 13|X|  + 15\\
&= -2\beta_1^2+ (2|X| + 1)\beta_1 -2|X| + 6.
\end{aligned}
$$
Let $g_{11}(x) = -2x^2 + (2|X| + 1)x  -2|X| + 6$ be a real function in $x$ for $x\in[\beta_2, +\infty)$. It is routine to check that the derivative function of $g_{11}(x)$ is
$$g_{11}'(x) = -4x + 2|X| + 1.$$
Thus, the unique solution of $g_{11}'(x) = 0$ is $\frac{2|X|+1}{4}$. Note that $\beta_2 =\frac{\sqrt{9|X|^2+2|X|+1}+|X|-1}{2}$. Obviously,
$$\frac{2|X| + 1}{4}
= \frac{|X|}{2}
+ \frac{1}{4}< \beta_2.$$
Therefore $g_{11}(x)$ is decreasing in the interval $[\beta_2, +\infty)$. Hence,
$$g_{11}(\beta_1) < g_{11}(\beta_2) = -4|X|^2 -\frac{5}{2}|X| +\frac{3}{2}\sqrt{9|X|^2 + 2|X| + 1}+ \frac{9}{2}.$$
Observe that
$$4^2|X|^2- (9|X|^2  + 2|X| + 1) = 7|X|^2  - 3|X|-1 > 0.$$
We have
$$g_{11}(\beta_1) < -4|X|^2  -\frac{5}{2}|X| + 6|X|+ \frac{9}{2}
= - 4|X|^2 + \frac{7}{2} |X| + \frac{9}{2}.$$
Let $g_{12}(x) =- 4x^2 + \frac{7}{2}x + \frac{9}{2}$ be a real function in $x$ for $x \in [2, +\infty)$. It is routine to check that the derivative function of $g_{12}(x)$ as
$$g_{12}'(x) = -8x + \frac{7}{2}.$$
Thus the unique solution of $g_{12}'(x) = 0$
is $\frac{7}{16}$. Together with
$$\frac{7}{16}< 2,$$
we have that $g_{12}(x)$ is decreasing in the interval $[2, +\infty)$. Then
$$g_{12}(x) \leq g_{12}(2) =-\frac{9}{2}.$$

 Thus, $f_1 < 0$ and  $\alpha(\beta_1) < 0$ when $|X| \geq 2$. Therefore, $\beta(\nu) \geq \beta_1 = \rho(G)$ when $\nu \geq 6$.
Next, we consider $\nu = 7$.
It is straightforward to check that $\beta(7) \approx 3.2731 < \frac{1+\sqrt{41}}{2}
\approx 3.7016.$  Hence,
we obtain a contradiction.

{\bf Case 2.} $G - X$ has no non-trivial connected component.

\quad\quad In what follows, we will show that  $\operatorname{iso}(G-X) \leq 2|X| + 2$.  Suppose that  $\operatorname{iso}(G-X) \geq 2|X| + 3$, then we have a new graph $H_{7}$ obtained from $G$ by adding an edge
in $I(G-X)$. Hence, we get $H_{7}-X$ has only one non-trivial connected component and $\operatorname{iso}(H_{7}- X)\geq 2|X| + 1$. By
Lemma \ref{lem444}, we obtain $\rho(G) < \rho(H_{7})$, contrary to the choice of $G$.

Note that $\operatorname{iso}(G-X) \geq 2|X|+1$. Next, we will consider $\operatorname{iso}(G-X) = 2|X|+1$  and $\operatorname{iso}(G-X) = 2|X|+2$.
We proceed by considering the following two subcases.

{\bf Subcase 2.1.} $\operatorname{iso}(G-X) = 2|X| + 1$. 

\quad\quad In this subcase, we have $G = K(|X|; (2|X| + 1)\times1)$ and $\nu = 3|X| + 1\geq4$. Hence, we obtain
$$\rho(G) =\beta_2 = \frac{\sqrt{9|X|^2 + 2|X| + 1}+|X| -1}{2}.$$

If $|X| = 2$, then $\nu = 7$ and $\rho(G) =\frac{1 +\sqrt{41}}{2}$.

If $|X| \neq 2$ ($\nu = 4$ or $\nu \geq 10$), then it is easy to see that 
$$\alpha(x) = x^3 - (3|X|-4)x^2-3(|X|x - 3|X|+4).$$
Next, we will show $\alpha(\rho(G)) \leq 0$. It is straightforward to check that $\alpha(\rho(G)) = \frac{(|X|-1)f_2}{2},$
where $f_2 = -8|X|^2 + 3\sqrt{9|X|^2 + 2|X| + 1}
+3|X| + 21.$

Obviously, if $|X| = 1$, then $\alpha(\rho(G)) = 0$ and  $\rho(G) \leq \beta(4)$. So in what follows, we consider 
$|X| \geq 3$, it suffices to prove $f_2 < 0$. 
Observe that
$$\sqrt{9|X|^2 + 2|X| + 1} =\sqrt {9(|X| + \frac{1}{9})^2 +\frac{8}{9}}<\sqrt {9(|X| + \frac{1}{9})^2} + 1.$$
Hence, we obtain
$$f_2 < -8|X|^2 + 9(|X| + \frac{1}{9}) + 3 + 3|X| + 21 = -8|X|^2 + 12|X| + 25.$$
Let $g_{13}(x) = -8x^2 + 12x + 25$ be a real function in $x$ for $x \in [3, +\infty)$. It is routine to check that the derivative function
of $g_{13}(x)$ is
$$g_{13}'(x) = -16x + 12.$$
Hence, the unique solution of $g_{13}'(x) = 0$
is $\frac{3}{4}$. Thus $g_{13}(x)$ is a decreasing in the interval $[3, +\infty)$. Therefore we get $g_{13}(x)\leq
g_{13}(3) = -11$. So, $f_2 < 0$ and  $\rho(G) < \beta(\nu)$ for $|X| \geq 3$, a contradiction.

{\bf Subcase 2.2.} $\operatorname{iso}(G-X) = 2|X| + 2$.

\quad\quad In this subcase, we get $G = K(|X|; (2|X| + 2)\times1)$ and $\nu = 3|X| + 2\geq5$. According
to the partition $V(G) = V(K_{|X|})\cup V((2|X|+2)K_1)$, the quotient matrix of  $A(G)$ is
$$M_{9} =\left(
       \begin{array}{cc}
         |X| - 1 & 2|X| + 2 \\
         |X|  & 0 \\
       \end{array}
     \right)
.$$
Then, the characteristic polynomial of $M_{9}$ is
$$\Phi_{M_{9}}(x) = x^2-(|X| -1)x -(2|X| + 2)|X|.$$
Since $V(G) = V(K_{|X|}) \cup V((2|X| + 2)K_1)$ is an equitable partition, together with  Lemma \ref{lem5}, 
we obtain that the largest root of $\Phi_{M_{9}}(x) = 0$, say $\beta_3$, equals $\rho(G)$. It is straightforward to check that 
$$\beta_3 = \rho(G) = \frac{\sqrt{9|X|^2 + 6|X| + 1}+|X|-1}{2}.$$
Together with $\nu = 3|X| + 2$, we have
$$\alpha(x) = x^3 -3 (|X|- 1)x^2 -(3|X|+1)x  + |X| - 9.$$
Next, we will show $\alpha(\rho(G)) \leq 0$. It is straightforward to check that  $\alpha(\rho(G)) = (|X|-1)\frac{f_3}{2}$, where
$$f_3 = -8|X|^2 + 3\sqrt{9|X|^2 + 6|X| + 1}- 5|X|  + 15.$$
Obviously, $\alpha(\rho(G)) = 0$ and  $\beta(\nu) \geq \rho(G)$ for $|X| = 1$.
So in what follows, we
consider  $|X| \geq 2$, it suffices to show  $f_3 < 0$.  Observe  that $4^2|X|^2 > 9|X|^2+6|X|+1.$
Thus, we obtain
$$f_3 <  -8|X|^2 + 7|X|  + 15.$$
Let $g_{14}(x) = -8x^2 + 7x + 15$ be a real function in $x$ for $x \in [2, +\infty)$. Then, the characteristic polynomial of $g_{14}(x)$ is
$$g_{14}'(x) = -16x + 7.$$
Thus, the unique solution of $g_{14}'(x) = 0$ is $\frac{7}{16}$. Obviously,
$$\frac{7}{16}< 2.$$
Hence, $g_{14}(x)$ is decreasing in the interval $[2, +\infty)$. So, $g_{14}(x) \leq g_{14}(2) = -3.$ 
Thus, $f_3 <0$ and $\rho(G) < \beta(\nu)$ for $|X| \geq2$, a contradiction to the condition.
This finishes the proof of Theorem \ref{theorem12}.

\end{proof}

	\item {\bf Distance spectral radius}

We give the proof of Theorem \ref{theorem13}, which gives a sufficient condition via the distance
spectral radius of a connected graph $G$ to ensure that $G$ is a strong star factor$\{K_{1,1}, K_{1,2}, C_{m}: m\geq3\}$-factor.

\quad\quad Now we give a proof of Theorem \ref{theorem13}.

\begin{proof}
  We prove it by contradiction. $G$ is not a $\{K_{1,1}, K_{1,2}, C_{m}: m\geq3\}$-factor. 
Then by Lemma \ref{lem333}, there
exists some nonempty subset $X\subseteq V(G)$ such that $\operatorname{iso}(G-X)\geq 2|X| + 1$. 

Suppose that $X=\emptyset$, we have $\operatorname{iso}(G)\geq 1$, which contradicts the condition that $G$ is a connected graph of $\nu \geq4$.

Choose a connected graph $G$ of order $\nu$ such that its distance
spectral radius  is as small as possible. Together with Lemma \ref{lem644} and the choice of $G$, the induced subgraph $G[X]$ and each connected component of
$G-X$ are complete graphs, respectively. Moreover, $G =G[X]\vee(G-X)$.

{\bf Fact 6.}  {$G-X$ has at most  one non-trivial connected component. }

\begin{proof}
Otherwise, adding edges among  all non-trivial connected component results in a larger order non-trivial connected component, a contradiction to the choice of $G$. This completes the proof of Fact 6.
\end{proof}

In what follows, we proceed by considering the two possible cases.

{\bf Case 1.}  $G_1$ is only one non-trivial connected component of $G-X$

\quad\quad In this case, let $|V(G_1)| =\nu_1\geq2$. In what follows, we will prove that  $\operatorname{iso}(G-X) =2|X| + 1$. Suppose that $\operatorname{iso}(G-X)\geq 2|X| + 2$, then we get a new graph $H_6$ obtained from $G$ by connecting every vertex of $G_1$ with one vertex of $I(G-X)$ by an edge.
Hence, we get that 
a proper
spanning subgraph of $H_6$ is $G$ and $\operatorname{iso}(H_6-X) \geq 2|X|+1$.  Together with Lemma \ref{lem644}, we obtain $ \mu_1(H_6)<\mu_1(G)$, contrary to the choice of $G$. Hence, we get  $\operatorname{iso}(G-X) \leq 2|X| + 1$.
Observe that $\operatorname{iso}(G-X) \geq 2|X|+1$. Then $\operatorname{iso}(G-X) = 2|X|+1$ and $G = K(|X|; \nu_1, (2|X|+1)\times1)$. Obviously,  $\nu = 3|X|+\nu_1+1 \geq 3(|X|+1) \geq 6$.
According to the equitable partition $V(G) = V(K_{|X|}) \cup V((2|X| + 1)K_1) \cup V(K_{\nu_1})$, the quotient
matrix of $D(G)$ is
$$M_{10} =\left(
       \begin{array}{ccc}
          |X| -1 & 2|X| + 1 & \nu - 3|X|- 1 \\
         |X| & 4|X| & 2(\nu- 3|X| - 1)\\
         |X| & 2(2|X| + 1) & \nu -3|X| - 2 \\
       \end{array}
     \right).
      $$
Hence the characteristic polynomial of $M_{10}$ is 
$$
\begin{aligned}\Phi_{M_{10}}(x)=&x^3 -(2|X| + \nu- 3)x^2 +(14|X|^2 - 4\nu|X|  +9|X| - 5\nu +6)x\\
&-6|X|^3 + 2\nu|X|^2+ 9|X|^2-3\nu|X|  + 10|X| -4\nu  + 4.
\end{aligned}
$$
Since $V(G) = V(K_{|X|}) \cup V((2|X| + 1)K_1) \cup V(K_{\nu_1})$ is an equitable partition, together with Lemma \ref{lem5}, we get that the largest root of $\Phi_{M_{10}}(x) = 0$, say $\gamma_1$, equals  $\mu_1(G)$.

 Next, we will prove the following fact which compares $\mu_1(G)$ with that of $K(1; \nu-4, 3\times1)$  at first.

{\bf Fact 7.} $\mu_1(G) \geq \mu_1(K(1; \nu-4, 3\times1))$ for $\nu \geq 3|X| + 3 \geq6$.

\begin{proof}
It is straightforward to check that   $G = K(1; \nu-4, 3\times1)$ for $|X| = 1$. Next, we only need
to show $\mu_1(G)>\mu_1(K(1; \nu-4, 3\times1))$ if $|X| \geq 2$.

According to the partition $V(K_1) \cup V(3K_1) \cup V(K_{\nu-4})$, the quotient
matrix of $D(K(1; (\nu-4, 3\times1))$ is
$$M_{11} =\left(
       \begin{array}{ccc}
          0 & 3 & \nu - 4 \\
         1 & 4 & 2(\nu- 4)\\
         1 & 6 & \nu -5 \\
       \end{array}
     \right).
      $$
Then, the characteristic polynomial of $M_{12}$ is
$$\Phi_{M_{11}}(x) = x^3-(\nu - 1)x^2 - (  9\nu  - 29)x - 5\nu  + 17.$$

Since $V(K_1) \cup V(3K_1) \cup V(K_{\nu-4})$ is an equitable partition, together
with Lemma \ref{lem5}, we have that the largest root of $\Phi_{M_{11}}(x) = 0$, say  $\gamma_2$, equals  $\mu_1(K(1; (\nu-4, 3\times1))$.
In what follows, we show $\Phi_{M_{10}}(\gamma_2) = \Phi_{M_{10}}(\gamma_2)-\Phi_{M_{11}}(\gamma_2) <0$ if $|X| \geq2$. It is straightforward to check that
$$\Phi_{M_{10}}(\gamma_2)-\Phi_{M_{11}}(\gamma_2) =(|X| - 1)[ - 2\gamma_2^2 + (14|X| -4\nu  + 23)\gamma_2-6|X|^2+ 3|X|+ \nu(2|X|-1) + 13].$$
Let $g_{15}(x) = - 2x^2 + (14|X| -4\nu  + 23)x-6|X|^2 + 3|X| + \nu(2|X|-1) + 13$ be a real function
in $x$. It is routine to check that the derivative function of $g_{15}(x)$ is
$$g_{15}'(x) = -4x + 14|X|  -4\nu   + 23.$$
Thus the unique solution of $g_{15}'(x) = 0$ is $-\nu + \frac{14|X|+23}{4}$. Together with
$$-\nu + \frac{14|X|+23}{4}< \nu + 2,$$
we obtain that $g_{15}(x)$ is a decreasing function in the interval $[\nu + 2, +\infty)$. According to Lemma \ref{lem9}, we get
$$\gamma_2= \mu_1(K(1; (\nu-4, 3\times1)) \geq \nu + 5 -\frac{18}{\nu}.$$
Note that $\nu \geq 3(|X| + 1) > 7$. Then
$$\gamma_2 > \nu + 5 - \frac{ 18 }{7}= \nu + 3  -\frac{4}{7}> \nu + 2.$$
Thus, we have $g_{15}(\gamma_2) <g_{15}(\nu + 2) = -6\nu^2 + (16|X| + 6)\nu -6|X|^2+ 31|X|  + 51.$ Let
$g_{16}(x) = -6x^2 + (16|X| + 6)x -6|X|^2 + 31|X| + 51$ be a real function in $x$ for $x\in[3(|X| + 1), +\infty)$.
It is routine to check that the derivative function of $g_{16}(x)$ is
$$g_{16}'(x) =  -2(6x - 8|X| - 3).$$
Therefore,
the unique solution of $g_{16}'(x) = 0$ is $\frac{ 8|X| + 3}{6}$. Together with
$$\frac{ 8|X| + 3}{6}< 3(|X| + 1),$$
we get that $g_{16}(x)$ is a decreasing  function in the interval $[3(|X| + 1), +\infty)$. According
to $\nu \geq 3(|X| + 1)$, we obtain
$$g_{16}(\nu) \leq g_{16}(3(|X| + 1)) = -12 |X|^2 -11|X| + 15.$$
Let $g_{17}(x) = -12 x^2 -11x + 15$ be a real function in $x$ for $x \in [2, +\infty)$. It is routine to check that the derivative function of $g_{17}(x)$ is
$$g_{17}'(x) = -24 x -11.$$
Thus
the unique solution of $g_{17}'(x) = 0$ is $-\frac{11}{24}$. Together with
$$-\frac{11}{24}< 2,$$
 we have that $g_{17}(x)$ is a decreasing  function in the interval $[2, +\infty)$. Hence, $g_{17}(x) \leq g_{17}(2) = -55 < 0.$ Therefore, $\Phi_{M_{10}}(\gamma_2)< 0$ and  $\mu_1(G) > \mu_1(K(1; (\nu-4, 3\times1))$ for $|X|\geq2$.
This completes the proof of Fact 7.
\end{proof}

Next, we consider $\nu = 7$.
By a simple calculation, we obtain $\mu_1(K(1; 3, 3\times 1))\approx 9.6974 > \mu_1(K(2; 5\times1) \approx 9.2170$. Hence, we get a contradiction to the condition for Case 1.

{\bf Case 2.} $G - X$ has no non-trivial connected component.

 \quad\quad In what follows, we will show that $\operatorname{iso}(G-X) \leq 2|X| + 2$. Suppose that $\operatorname{iso}(G-X) \geq 2|X| + 3$, then we get a new graph $H_{7}$ obtained from $G$ by adding an edge
in $I(G-X)$. Thus, we obtain $H_{7}-X$ has only one non-trivial connected component and $\operatorname{iso}(H_{7}- X)\geq 2|X| + 1$. Hence, we have
$|E(G)| < |E(H_{7})|$, contrary to the choice of $G$.
 Note that $\operatorname{iso}(G-X) \geq 2|X|+1$. Next, we will consider $\operatorname{iso}(G-X) = 2|X|+1$  and $\operatorname{iso}(G-X)= 2|X|+2$.

{\bf Subcase 2.1.} $\operatorname{iso}(G-X) = 2|X| + 1$. 

\quad\quad In this subcase $G = K(|X|; (2|X| + 1)\times1$ and $\nu = 3|X|+1 \geq 4$. According to the partition $V(G) = V(K_{|X|}) \cup V((2|X| + 1)K_1)$, the quotient matrix of $D(G)$ equals
$$M_{12} =\left(
       \begin{array}{cc}
          |X| -1 & 2|X| + 1  \\
         |X| & 4|X| \\
       \end{array}
     \right).
      $$
Then, the characteristic polynomial of $M_{12}$ is
$$\Phi_{M_{12}}(x) = x^2- (5|X| - 1)x + 2|X|^2 - 5|X|.$$
Since $V(G) = V(K_{|X|})\cup V((2|X|+1)K_1)$ is an equitable partition, by Lemma \ref{lem5}, we have that the largest root of $\Phi_{M_{12}}(x) = 0$, say $\gamma_3$,
equals $\mu_1(G)$.
Hence,
 $$\Phi_{M_{11}}(x) = x^3-3|X|x^2- (27|X|-20)x -15|X| + 12.$$
 It is routine to check that $G = K(1; 3\times1)$ and $\nu = 4$  for $|X| = 1$. If $|X| = 2$,
then $G = K(2; 5\times1)$ and $\nu = 7$. If $|X| = 3$, then $\nu = 10$ and $\mu_1(G) = 7+\sqrt{46}\approx 13.7823 > \mu_1(K(1; 6, 3\times1))\approx13.6470$. We are to show $\mu_1(G) > \mu_1(K(1; \nu-4, 3\times1))$ and $\nu\geq13$ when $|X| \geq 4$.
Let $g_{18} (x) = x\Phi_{M_{12}}(x)-\Phi_{M_{11}}(x)$. Hence, we have
$$g_{18} (x) = -(2|X| - 1)x^2 + 2(|X|^2 + 11|X| -10)x + 3(5|X|- 4).$$
 It is routine to check that the derivative function of  $g_{18}(x)$ is
$$g_{18}'(x) = (-4|X| + 2)x + 2(|X|^2 + 11|X| - 10).$$
Thus,
the unique solution of $g_{18}'(x) = 0$ is $\frac{|X|^2+11|X|-10}{2|X|-1}$. Together with
$$\frac{|X|^2 + 11|X| -10}{2|X|- 1}< 3|X| + 4.$$
we get that $g_{18}(x)$ is a decreasing function in the interval $[3|X| + 4, +\infty)$. Note that $\gamma_2 \geq \nu + 5 -\frac{18}{\nu}$ and $\nu\geq 13$. Then
$$\gamma_2 \geq3|X| + 6 - \frac{18}{13}> 3|X| + 4.$$
Hence $g_{18}(\gamma_2) < g_{18}(3|X| + 4) = -12|X|^3 + 35|X|^2 + 35|X| - 76$. Let $g_{19}(x) = -12x^3 + 35x^2 + 35x-76$ be a real function in $x$ for
$x \in [4, +\infty)$. It is routine to check that the derivative function of $g_{19}(x)$ is
$$g_{19}'(x) = -36x^2 + 70x + 35 =-36(x- \frac{35-\sqrt{2485}}{36})(x- \frac{35 +\sqrt{2485}}{36}).$$
Obviously, $\frac{35-\sqrt{2485}}{36}< \frac{35 +\sqrt{2485}}{36}< 4$. Thus  $g_{19}(x)$ is a decreasing function in the interval $[4, +\infty)$. So $g_{19}(x) \leq g_{19}(4) = -144 < 0$. Therefore,
$\mu_1(G) > \mu_1(K(1; \nu-4, 3\times1))$, a contradiction.

{\bf Subcase 2.2.} $\operatorname{iso}(G-X) = 2(|X|+ 1)$.

\quad\quad In this subcase $G = K(|X|; (2(|X| + 1))\times1)$ and $\nu = 3|X| + 2$. According to the partition $V(G) =
V(K_{|X|}) \cup V(2(|X| + 1)K_1)$), the quotient matrix of  $D(G)$ is
$$M_{13} =\left(
       \begin{array}{cc}
          |X| -1 & 2|X| + 1  \\
         |X| & 2(2|X| + 1) \\
       \end{array}
     \right).
      $$
Then, the characteristic polynomial of  $M_{13}$ is
$$\Phi_{M_{13}}(x) = x^2 - (5|X|  + 1)x + 2|X|^2 - 4|X|- 2.$$
Since $V(G) = V(K_{|X|})\cup V((2|X|+2)K_1)$ is an equitable partition, by Lemma \ref{lem5}, the largest root of $\Phi_{M_{13}}(x) = 0$, say $\gamma_4$,
equals $\mu_1(G)$.
Hence, we obtain
 $$\Phi_{M_{11}}(x) = x^3- (3|X| + 1)x^2 - (27|X| - 11)x - 15|X| + 7.$$
Next, we are to show $\mu_1(G) \geq\mu_1(K(1; 3|X|-2, 3\times1))$.
Obviously, $G =K(1; 1, 3\times1)$ for $|X| =1.$ If $|X|=2$, then $G =K(2; 6\times1)$ and $\nu =8$. By a simple calculation, we get
$\mu_1(G) \approx 11.1789 > \mu_1(K(1; 4, 3\times1) ) \approx 11.0715$. Hence it suffices to prove $\gamma_2\Phi_{M_{13}}(\gamma_2) - \Phi_{M_{11}}(\gamma_2) < 0$ for $|X| \geq 3$. Note that
$$\gamma_2\Phi_{M_{13}}(\gamma_2) - \Phi_{M_{11}}(\gamma_2) = -2|X|\gamma_2^2 + (2|X|^2 + 23|X| - 13)\gamma_2 + 15|X| - 7.$$
Let $g_{20}(x) := -2|X|x^2 + (2|X|^2 + 23|X| - 13)x + 15|X| - 7$ to be a real function in $x$.  It is routine to check that the derivative function of  $g_{20}(x)$ is
$$g_{20}'(x) = -4|X|x + 2|X|^2 + 23|X| - 13.$$
Clearly, the unique solution of $g_{20}'(x) = 0$ is $\frac{|X|}{2}+ \frac{23}{4}-\frac{13}{4|X|}$ . Together with
$$\frac{|X|}{2}+ \frac{23}{4}-\frac{13}{4|X|}< 3|X| + 5,$$
we have that $g_{20}(x)$ is a decreasing function in the interval $[3|X| + 5, +\infty)$. Note that $\gamma_2 \geq \nu + 5 -\frac{18}{\nu}$ and $\nu = 3|X| + 2 \geq 11$. Then
$$\gamma_2\geq 3|X| + 7 - \frac{18}{11} > 3|X| + 5.$$
Hence, $g_{20}(\gamma_2) < g_{20}(3|X| + 5) = -12|X|^3 + 19|X|^2 + 41|X|-72.$ Let $g_{21}(x) = -12x^3 + 19x^2 + 41x -72$ be a real function in $x$
for $x \in [3, +\infty)$. The derivative function of $g_{21}(x)$ is
$$g_{21}'(x) = -36x^2 + 38x + 41 = -36(x-\frac{19 -\sqrt{1837}}{36})(x - \frac{19 +\sqrt{1837}}{36}).$$
Obviously,
$$\frac{19 -\sqrt{1837}}{36}< \frac{19 +\sqrt{1837}}{36}< 3.$$
Thus $g_{21}(x)$ is a decreasing function in the interval $[3, +\infty)$. Hence $g_{21}(x) < g_{21}(3) = -102 < 0$. So, $\mu_1(G) > \mu_1(K(1; 3|X|-2, 3\times1))$
for $|X| \geq 3$, a contradiction.
This finishes the proof of Theorem \ref{theorem13}.

\end{proof}

\end{enumerate}

\section{Concluding remarks}

\begin{enumerate}
	\item 
At last we prove the bounds obtained in Theorems \ref{theorem11}-\ref{theorem13} are best possible.

\begin{theo}\label{theorem14}
Let $\nu \geq4$ be a integer. Then  
$K(2; 5\times1)$ and $K(1; \nu-4, 3\times1)$ have no a $\{K_{1,1}, K_{1,2}, C_{m}: m\geq3\}$-factor.
\end{theo}
\begin{proof}
Let $u$ be the vertex with the maximum degree of $K(1; \nu-4, 3\times1)$.
 Let $X= \{u\}$ (resp. $V(K_2)$), then $\operatorname{iso}(K(1; \nu-4, 3\times1)-X) \geq 3 > 2=2|X|$ (resp. $\operatorname{iso}(K(2; 5\times1)- X) = 5 >4= 2|X|$). According to Lemma \ref{lem333}, we obtain $K(1; \nu-4, 3\times1)$ and $K(2; 5\times1)$  have no a $\{K_{1,1}, K_{1,2}, C_{m}: m\geq3\}$-factor. This completes the proof of Theorem \ref{theorem14}.
\end{proof}

Note that $|E(K(2; 5\times1))| = 11$. According to Theorem \ref{theorem14},
we get the bounds obtained in Theorems \ref{theorem11} are best possible.
\begin{theo}\label{theorem15}
Let $\nu \geq4$ be a integer. Then 

 $\operatorname{(i)}$ $\rho(K(2; 5\times1))=\frac{1+\sqrt{41}}{2}$.
 
 $\operatorname{(ii)}$ $\rho(K(1; \nu-4, 3\times1))$ is equal to the largest root of $x^3 -(\nu- 5)x^2-(\nu -1)x+3\nu-15 = 0.$

\end{theo}
\begin{proof}
$\operatorname{(i)}$ According to the partition $V(K(2; 5\times1) = V(K_2) \cup V(5K_1)$, the
quotient matrix of $A(K(2; 5\times1))$ equals
$$M_{14} =\left(
       \begin{array}{cc}
          1 & 5 \\
         2 & 0 \\
       \end{array}
     \right).
      $$

Then, the characteristic polynomial of $M_{14}$ is
$$\Phi_{M_{14}}(x) = x^2- x-10.$$
By Lemma \ref{lem5}, we obtain that $\rho(K(2; 5\times1))=\frac{1+\sqrt{41}}{2}$ is equal to the largest root of $ x^2- x-10= 0$.

$\operatorname{(ii)}$ According
to the partition $V(K(1; \nu-4, 3\times1)) = V(K_1)\cup V(K_{\nu-4})\cup V(3K_1)$, the quotient matrix of $A(K(1; \nu-4, 3\times1))$ equals
$$M_{15} =   \left(
  \begin{array}{ccc}
    0 & 3 & \nu-4 \\
    1 & 0 &  0\\
    1  & 0 &  \nu-5 \\
  \end{array}
     \right).
      $$
Then, the characteristic polynomial of $M_{15}$ is
$$\Phi_{M_{15}}(x) =x^3 -(\nu- 5)x^2-(\nu -1)x+3\nu-15.$$

By Lemma \ref{lem5}, one has that $\rho(K(1; \nu-4, 3\times1))$ is equal to the largest root of $x^3 -(\nu- 5)x^2-(\nu -1)x+3\nu-15 = 0.$
This completes the proof of \ref{theorem15}.

\end{proof}
\quad\quad By Theorems \ref{theorem14} and \ref{theorem15}, we get  the bounds in Theorem \ref{theorem12} are best possible.
According to the proof of Theorem \ref{theorem14}, we have  $K(2; 5\times1)$  and $K(1; \nu-4, 3\times1)$ have no
a $\{K_{1,1}, K_{1,2}, C_{m}: m\geq3\}$-factor. Hence the condition in Theorem \ref{theorem13} can not be  replaced by the
condition that $\mu_1(G)\leq\mu_1(K(2; 5\times1))$ and $\mu_1(G)\leq\mu_1(K(1; \nu-4, 3\times1))$ which implies the bounds established
in Theorem \ref{theorem13} are  the best possible.
\end{enumerate}
\begin{enumerate}

\item  Professor Shuchao Li and his team provided us with help and guidance. We would like to express our sincere gratitude.

\item {\bf Declaration}

{\bf Conflict of interest}
The authors declare that they have no conflicts of interest
to this work.

\item {\bf Data Availability Statement}

 No data was used for the research described in the article.

\item {\bf Acknowledgements}

 The authors would like to thank the anonymous reviewers for their careful reading of our manuscript and their many
insightful comments and suggestions. This work is supported by the National Science Foundation of China (Nos.12261074 and 12461065).

\end{enumerate}

\end{document}